\documentclass[12pt]{amsart}

\usepackage{amsfonts, amssymb, amsthm, amsmath, bm, braket, enumitem, hyperref, mathtools, comment, mathrsfs}
\usepackage[all]{xy}
\usepackage{stmaryrd}
\usepackage{graphicx}
\usepackage{amscd}

\newtheorem{theorem}{Theorem}[section]

\newtheorem{corollary}{Corollary}[theorem]

\newtheorem{proposition}[theorem]{Proposition}
\newtheorem{prop}[theorem]{Proposition}
\newtheorem{example}[theorem]{Example}

\theoremstyle{definition}
\newtheorem{definition}[theorem]{Definition}
\newtheorem{defn}[theorem]{Definition}
\newtheorem{remark}[theorem]{Remark}

\newtheorem{question}[theorem]{Question}
\newtheorem{quest}[theorem]{Question}

\theoremstyle{remark}
\newtheorem*{obsrem}{Remark}

\newcounter{casenum}

\newcommand{\defi}[1]{\textsf{#1}} 

\makeatletter
\def\namedlabel#1#2{\begingroup
    #2%
    \def\@currentlabel{#2}%
    \phantomsection\label{#1}\endgroup
}
\makeatother

\newcommand{\bbC}{\mathbb{C}}

\newcommand{\bbN}{\mathbb{N}}

\newcommand{\bbQ}{\mathbb{Q}}
\newcommand{\bbR}{\mathbb{R}}

\newcommand{\bbZ}{\mathbb{Z}}

\newcommand{\calF}{\mathcal{F}}
\newcommand{\calG}{\mathcal{G}}

\newcommand{\calM}{\mathcal{M}}

\newcommand{\calQ}{\mathcal{Q}}

\newcommand{\calT}{\mathcal{T}}

\newcommand{\frakA}{\mathfrak{A}}

\newcommand{\frakS}{\mathfrak{S}}

\newcommand{\prm}{^\prime}
\newcommand{\parent}[1]{\left( #1 \right)}

\renewcommand{\i}{\text{\rm i}}
\newcommand{\inv}{^{-1}}
\newcommand{\abs}[1]{\left|#1\right|}
\DeclareMathOperator{\into}{\hookrightarrow}

\DeclareMathOperator*{\MJoin}{\bigtriangledown}

\newcommand{\Sum}{\frakS}
\newcommand{\Nor}{N}

\newcommand{\Domain}{{\rm D}}
\newcommand{\Codomain}{{\rm E}}
\newcommand{\Sums}[2]{\textbf{\upshape{S}}\parent{#1,#2}} 
\newcommand{\PSums}[2]{\textbf{\upshape{pMS}}\parent{#1,#2}} 
\newcommand{\WSums}[2]{\textbf{\upshape{wMS}}\parent{#1,#2}} 
\newcommand{\MSums}[2]{\textbf{\upshape{MS}}\parent{#1,#2}} 
\newcommand{\R}{{\rm R}} 

\newcommand{\TSumCat}{\textbf{\upshape{TSum}}}
\newcommand{\QSumCat}{\textbf{\upshape{QSum}}}
\newcommand{\SumCat}{\textbf{\upshape{Sum}}}
\newcommand{\pMSumCat}{\textbf{\upshape{pMSum}}}
\newcommand{\wMSumCat}{\textbf{\upshape{wMSum}}}
\newcommand{\MSumCat}{\textbf{\upshape{MSum}}}

\newcommand{\Reg}[1]{{\rm Reg}\parent{#1}}

\newcommand{\T}{\calT} 
\newcommand{\M}{\calM} 
\newcommand{\Q}{\calQ} 

\newcommand{\Add}{\frakA} 

\newcommand{\Abelsum}{\Sum_A} 
\newcommand{\Abeldomain}{\Domain_A}

\newcommand{\Borelsum}{\Sum_B} 
\newcommand{\Boreldomain}{\Domain_B}

\newcommand{\Convsum}{\Sum_c} 
\newcommand{\Convdomain}{\Domain_c} 

\newcommand{\Abssum}{\Sum_a} 
\newcommand{\Absdomain}{\Domain_a} 

\newcommand{\CHsum}{\Sum_{CH}} 
\newcommand{\CHdomain}{\Domain_{CH}} 

\newcommand{\Eulersum}[1]{\Sum_{E,#1}}
\newcommand{\Eulerdomain}[1]{\Domain_{E,#1}}
\newcommand{\Eulersumspecial}{\Sum_{E}}
\newcommand{\Eulerdomainspecial}{\Domain_{E}}

\newcommand{\psum}{\Sum_p} 

\newcommand{\ignore}[1]{}

\begin{document}
\title{Telescopic, Multiplicative, and Rational Extensions of Summations} 
\author{Robert J. MacG. Dawson}\email{rdawson@cs.smu.ca}
\author{Grant Molnar}\email{Grant.S.Molnar.GR@dartmouth.edu}\thanks{The second author received support from the Gridley Fund for Graduate Mathematics during the writing of this paper.}

\begin{abstract}
       A summation is a shift-invariant $\R$-module homomorphism from a submodule of $\R[[\sigma]]$ to $\R$ or another ring. \cite{Dawson} formalized a method for extending a summation to a larger domain by telescoping. In this paper, we revisit telescoping, we study multiplicative closures of summations (such as the usual summation on convergent series) that are not themselves multiplicatively closed, and we study rational extensions as a generalization of telescoping.
\end{abstract}

\maketitle

\section{Introduction}\label{Section: Introduction}

Let $\R$ be a commutative unital ring with $0 \neq 1$. A \defi{(formal) series} over $\R$ is a sequence of terms $a_0 + a_1 + \dots$ with $a_0, a_1, \ldots \in \R$, indexed by the natural numbers $\bbN = \Set{0, 1, 2, \dots}$. The collection of all such series forms an $\R$-module under pointwise addition and scalar multiplication. Of course, a series should be more than ``a sequence with plus signs''! We will interpret our series as formal power series in a variable $\sigma$, as justified below.

Our notation makes it natural to identify $r \in \R$ with the series $r + 0 + 0 + \dots$, and we do so for the remainder of the paper. Thus 
\[
	\parent{r + 0 + 0 + \dots} \cdot \parent{a_0 + a_1 + a_2 + \dots} = r a_0 + r a_1 + r a_2 + \dots.
\] 
Now how shall we define $\parent{0 + r + 0 + \dots} \cdot \parent{a_0 + a_1 + a_2 + \dots}$? It is natural to declare
\[
	\parent{0 + r + 0 + \dots} \cdot \parent{a_0 + a_1 + a_2 + \dots} \coloneqq 0 + r a_0 + r a_1 + r a_2 + \dots.
\]
More generally, define the right-shift operator $\sigma$ on the space of formal series by 
\[
	\sigma : a_0 + a_1 + \dots \mapsto 0 + a_0 + a_1 + \dots;
\]
we enforce commutativity between $\sigma$ and the elements of $\R$, so that $\sigma r = r \sigma$ for each $r \in R$. Then we may formally write 
\[
	a_0 + a_1 + a_2 + \dots = a_0 +  \sigma a_1 +  \sigma^2 a_2 + \dots = a_0 + a_1 \sigma + a_2 \sigma^2 + \dots,
\]
and the formal series act on one another via the Cauchy product. This approach naturally identifies the space of formal series with the $\R$-algebra $\R[[\sigma]]$ of formal power series over $\R$.

This identification is natural in another sense. In classical analysis, the projection 
\[
	\sum_{n = 0}^\infty a_n \mapsto \sum_{n = 0}^N a_n
\] 
is a key ingredient in defining the sum of a convergent series. But 
\[
\R[[\sigma]] \coloneqq \lim_{\leftarrow} R[\sigma]/(\sigma^n)
\]
is precisely the $\R$-algebra that maps universally down to each of these spaces of finite sums (construed with appropriate multiplication). For these reasons, we deem it reasonable to view $\R[[\sigma]]$ as the space of all series over $\R$. If the series $\sum_{n = 0}^\infty a_n z^n \in \bbC[[z]]$ converges to a holomorphic function $A(z)$ in a neighborhood of the origin, we write $A(\sigma)$ for the formal series $\sum_{n = 0}^\infty a_n \sigma^n$. 

While formal power series are a fruitful object of study in their own right \cite{Ahmed-Sana, Aroca-Rond, Enochs-Jenda-Ozbek, Fripertinger-Reich, Hizem, Ivan, Roman}, mathematicians have a longstanding interest in assigning a definite ``sum'' to each series, or at least as many series as possible \cite{Balser, Beekman-Chang1, Beekman-Chang2, Boos, Cauchy, Euler, Faulstich-Luh, Hardy, Hill, Remy, Shawyer-Watson, Calculus}. For a finitely supported series $A$, this is straightforward: if $A \in \R[\sigma]$, we set $\Add(A) \coloneqq A(1)$, and say $\Add(A)$ is the \defi{sum} of $A$. Thus for $A = a_0 + \dots + a_N + 0 + 0 + \dots$, we have $\Add(A) = a_0 + \dots + a_N \in R$, as intuition demands. This definition works equally well if we wish to sum $A$ in an $\R$-algebra $\Codomain$.

The question of how to define the sum of
\[
	X = \sum_{n = 0}^\infty x_n \sigma^n = x_0 + x_1 + x_2 + \dots
\] 
becomes much more delicate when $X$ has infinite support. If $X$ has real or complex coefficients, the \defi{classical sum} $\Convsum\parent{X}$ of $X$ is defined by
\begin{equation}\label{Equation: Classical summation}
	\Convsum\parent{X} \coloneqq \lim_{N \to \infty} \sum_{n = 0}^N x_n,
\end{equation}
wherever this limit exists. We write $\Convdomain$ for the vector space of convergent series over $\bbC$. Indeed, \eqref{Equation: Classical summation} makes sense whenever the coefficients of $X$ reside in a field with an absolute value. In particular, if $X$ has $p$-adic coefficients, we say $\Sum_p(X)$ is the limit of the partial sums of $X$ in the field $\bbQ_p$ (or any pertinent extension thereof). We write $\Domain_p$ for the vector space of (absolutely) convergent series over $\bbQ_p$.

Numerous methods exist for defining the sum of a series $X$ with real or complex coefficients when classical summation fails. We record a non-exhaustive sampling below (for additional summation methods, see \cite{Hardy}).

\begin{itemize}
	\item The \defi{Ces\'{a}ro-H\"{o}lder sum} $\CHsum(X)$ of $X$ is defined by
\[
		\CHsum(X) \coloneqq \lim_{N \to \infty} \frac{k!}{N^k} \sum_{n = 0}^N {{N-n+k} \choose {k}} x_{n}
\]
for any $k$ where this limit exists \cite[page 96]{Hardy}. If the limit exists for fixed $k$, we say that $X$ is \defi{Ces\'{a}ro-H\"{o}lder summable} of order $k$. We write $\Domain_{CH}$ for the vector space of Ces\'{a}ro-H\"{o}lder summable series. \\ 
		\item The \defi{Abel sum} $\Abelsum(X)$ of $X$ is defined by
\[
		\Abelsum(X) \coloneqq\lim_{\rho \to 1^{-}} \sum_{n = 0}^\infty x_n \rho^n
\]
wherever this limit exists \cite[page 7]{Hardy}. We write $\Domain_A$ for the vector space of Abel summable series.\\
		\item The \defi{exponential Borel sum} is defined by 
\[
		\Sum_b(X) \coloneqq \lim_{t \to \infty} e^{-t} \sum_{n = 0}^\infty \parent{\sum_{m = 0}^n x_m} \frac{t^n}{n!}
\]
wherever this limit exists \cite[page 27]{Shawyer-Watson}. The closely related \defi{integral Borel sum} $\Sum_b\prm(X)$ of $X$ is defined by
\[
		\Sum_b\prm(X) \coloneqq \int_0^\infty e^{-t} \sum_{n = 0}^\infty \frac{x_n}{n!} t^n dt
\]
wherever this limit exists \cite[page 27]{Shawyer-Watson}. It is known that $\Sum\prm_b(X)=x$ if and only if $\Sum_b(\sigma X)=x$ \cite[Theorem 126]{Hardy}. We write $\Domain_b$ for the vector space of exponential Borel summable series, and $\Domain_b\prm$ for the vector space of integral Borel summable series. \\
		\item Let $\Omega \subseteq \bbC$ be a connected set containing 0 with 1 as a limit point. The \defi{Euler sum} $\Eulersum \Omega(X)$ of $X$ (relative to $\Omega$) is defined by
\[
		\Eulersum \Omega (X) \coloneqq \lim_{\substack{\rho \to 1 \\ \rho \in \Omega}} \xi(\rho)
\]
where $\xi(z)$ is the analytic continuation of $\sum_{n = 0}^\infty x_n z^n$ on an open set $U \supseteq \Omega$, if such an analytic continuation exists for some $U$ (This definition is adapted from \cite[page 7]{Hardy}). We write $\Domain_{E, \Omega}$ for the vector space of Euler summable series (relative to $\Omega$). We also write $\Eulersumspecial \coloneqq \Eulersum {[0,1)}$ and $\Domain_E \coloneqq \Domain_{E, [0,1)}$. Euler has other summations associated to his name (see \cite{Euler} and \cite[Chapters 8, 13]{Hardy}), but we have no need of them in this paper.

\end{itemize}

A few points bear immediate mention. Observe that each of these methods carries a distinctly analytic flavor; by contrast, the methods we propound in this paper are entirely algebraic. In any case, given a series $X$, we may hope these definitions will allow us to coherently discuss the sum of $X$, even if $X$ were divergent. But we are confronted at once with three issues. 

Our first issue is that the value of the sum of $X$ may depend on our choice of summation method. In many cases one summation merely extends another: the domain of the first includes the domain of the second, and they agree there. If this were always the case, there would be a ``maximal'' or ``Platonic'' summation, of which other methods could compute a greater or a lesser part. But, as various examples below show, this is not the case.

Our second issue is that these definitions presuppose we are summing series in a field with absolute value. But addition is possible in every ring, and so it is natural to ask how we might define infinite sums in every ring as well. For instance, how could we define $1 + 2 + 4 + 8 + \dots$ in $\bbZ / 6 \bbZ$? This question is slightly ambiguous; every ring (with identity) contains a canonical image of $\bbZ$.  We may consider $1 + 2 + 4 + 8 + \dots$ as a series in $\bbZ / 6 \bbZ$ (where $2=8$) or in $\bbZ$. In the latter case, we would be summing series over one ring ($\bbZ$) into another ring ($\bbZ / 6 \bbZ$).
 
Even the traditional definition of infinite summation forces us to consider sums that ``escape'' from $\bbQ$ into $\bbR$; and formal telescoping lets sums ``escape'' from $\bbZ$ into $\bbQ$ \cite[Section 3]{Dawson}. When can we sum series over $\R$ into some other $\R$-algebra?  The following examples showcase some of the subtleties that arise.

\begin{example} \cite[page 82]{Koblitz}\label{Example: series sums to different values}
	Define 
	\[
		X \coloneqq \sqrt{1 + \frac{7}{9} \sigma} = \sum_{n = 0}^\infty {{1/2} \choose n} \frac{7^n}{9^n} \sigma^n = 1 + \frac{7}{18} - \frac{49}{648} + \frac{343}{11664} - \frac{12005}{839808} + \frac{117649}{15116544} - \dots;
	\] 
	we easily compute that the sum of $X$ in $\bbR$ is $\Convsum(X) = \frac{4}{3}$, and the sum of $X$ in $\bbQ_7$ is $\Sum_7(X) = - \frac{4}{3}$. Thus the sum of $X$ as a series over $\bbQ$ may (at least) be taken to be either $\frac{4}{3}$ or $-\frac{4}{3}$, depending on context and convenience. We note that both of these are sums  in the classical sense of limits of partial sums; which one we get depends on our choice of metric.
\end{example}

Example \ref{Example: series sums to different values} is by no means unique. Indeed, for any integer $a \geq 3$, if we set 
\[
X_a \coloneqq \sqrt{1 + \parent{\frac{2a+1}{a^2}} \sigma} = \sum_{n = 0}^\infty {{1/2} \choose n} \parent{\frac{2a+1}{a^2}}^n \sigma^n,
\]
then $\Convsum\parent{X_a} = \frac{a+1}{a}$, whereas $\Sum_p\parent{X_a} = -\frac{a+1}{a}$ for each prime $p$ dividing $2a + 1$. More generally, for any positive integer $m$ we can find a rational series $X$ and a set of primes $\{p_1,p_2,\dots,p_m\}$ such that the sum of $X$ in each $\bbQ_{p_i}$ is different. 

\begin{example} 
 Let $r$ be a rational number, and let $Y_r \coloneqq \sum_{k=0}^\infty \binom{1/2}{n}(r^2-1)^n \sigma^k$, so that $Y_r^2 = 1+(r^2-1)\sigma$. As all the $p$-adic limit summations $\Sum_p$ preserve products (see Section \ref{Section: Multiplicative Summations}), if $Y_r$ converges $p$-adically, it converges to $\pm r$. 

Let $p$ be an odd prime dividing $r \pm 1$; then $Y_r$ converges $p$-adically, and all of its terms past the first are strictly smaller than 1. As the $p$-adic metric is an ultrametric, $\Sum_p(Y_r)$ is whichever of $r$ and $-r$ is $p$-adically closer to 1, so $\psum(Y_r) = \mp r$. 

Now let $p_1,\dots,p_m$ be an arbitrary sequence of odd primes. By the Chinese Remainder Theorem, we may choose a sequence of integers $r_1,\dots,r_m$ such that $r_i \equiv -1 \pmod {p_j}$ if $i=j$, and $r_i \equiv 1 \pmod {p_j}$ otherwise. Let 
\[
Y \coloneqq \sum_{n=0}^{\infty} \binom{1/2}{n} \left ((r_1^2-1)^n +\dots + (r_m^2-1)^n\right) \sigma^n.
\]
By construction, $\Sum_{p_j} (Y) = 2r_j -(r_1 +\dots+ r_m)$ for each $1 \leq j \leq m$. If the sequence $r_1,\dots,r_m$ grows sufficiently rapidly, each of these sums is distinct.
\end{example}

\begin{example} \cite[page 16]{Hardy} Recall that $n!! \coloneqq \prod_{0 \leq 2k < n} (n - 2k) = n \cdot (n - 2) \cdot \dots$. The series 
\[
\sum_{n=0}^\infty \frac{(2n-1)!!}{(2n)!!}\left(\frac{2z}{1+z^2}\right)^{2n}
\]
converges to $\frac{1+z^2}{1-z^2}$ for $\abs{z}<1$ but to $\frac{1+z^2}{z^2-1}$ for $\abs{z}>1+\sqrt{2}$. When $z=2i$ we obtain the series $1+\frac{8}{9} + \frac{32}{27}+\cdots$. This is divergent: Euler's method would evaluate it to $3/5$, but a modification of Euler's method with basepoint $z_0$, $\abs{z_0} > 1+\sqrt{2}$, would assign the value $-3/5.$
\end{example}

Our third issue is that we do not yet a principled notion of what a summation \emph{is}, or what properties it is supposed to have. Indeed, various notions of summation exist in the literature. For instance, for Boos, a summation is a linear transformation on the space of sequences, post-composed with a limit functional \cite[Formulation 1.2.10]{Boos}. But we will follow in Hardy's footsteps. He lists the following axioms \cite[page 6]{Hardy} (Note that these axioms assert existence of sums as well as equality):
\begin{enumerate}[label=(\Alph*)]
  \item \label{HardyA} If $\Sum \parent{a_0+a_1+a_2+\dots} = s$ then 
  \[
  \Sum \parent{ka_0+ka_1+ka_2+\dots} = ks;
  \]
  \item \label{HardyB} If $\Sum \parent{a_0+a_1+a_2+\dots} = s$ and $\Sum \parent{b_0+b_1+b_2+\dots} = t$, then 
  \[
  \Sum \parent{(a_0 + b_0) + (a_1 + b_1) + (a_2 + b_2) + \dots}= s+t;
  \]
  \item \label{HardyC} If $\Sum(a_0+a_1+a_2+\dots) = s$, then 
  \[
  \Sum(a_1+a_2+a_3 + \dots) = s-a_0,
  \]
  and conversely.
  \end{enumerate}
Hardy states that every summation method in the book obeys axioms \ref{HardyA} and \ref{HardyB}, while most obey axiom \ref{HardyC}. In \cite[Theorem 40]{Hardy}, he divides axiom \ref{HardyC} into two subaxioms:

\begin{itemize}
\item[\namedlabel{Hardy subaxiom 1}{($\gamma$)}] If $\Sum(a_1+a_2+a_3\cdots) = s-a_0$, then 
 \[
 \Sum(a_0+a_1+a_2+\cdots) = s;
 \]
\item[\namedlabel{Hardy subaxiom 2}{($\delta$)}] If $\Sum(a_0+a_1+a_2+\cdots) = s$, then 
\[
\Sum(a_1+a_2+a_3\cdots) = s-a_0.
\]
\end{itemize}

Hardy implicitly assumes the coefficients and sum of a series to be in $\bbC$. At no point does he go outside $\bbC$, and he rarely remarks upon summations that ``escape'' from a proper subring, e.g. by summing $1-1+1-1+\cdots$ to $\frac{1}{2}$. He does however say that a process summing a real series to a complex value has ``an air of paradox'' \cite[page 16]{Hardy}. 
 
As above, let $\R$ be a commutative unital ring with $0 \neq 1$, and let $\Codomain$ be a commutative unital $\R$-algebra with $0 \neq 1$. Let $\Domain$ be a set and $\Sum : \Domain \to \Codomain$ a map. 

\begin{defn}
	The tuple $(\R, \Codomain, \Domain, \Sum)$ is a \defi{summation on $\R$ to $\Codomain$} if it satisfies the following axioms:
	\begin{enumerate}[label=(\Roman*)]
		\item We have $\R[\sigma] \subseteq \Domain \subseteq \R[[\sigma]]$;\label{Condition: AreSeries}
		\item The set $\Domain$ is an $\R$-module, and the map $\Sum$ is an $\R$-module homomorphism; \label{Condition: IsModuleMorphism}
		\item We have $\Sum(1) = 1$; \label{Condition: ExtendsAdd}
		\item We have $(1 - \sigma) \Domain \subseteq \Domain$, and the morphism $\Sum$ factors through $\Domain / (1 - \sigma) \Domain$. \label{Condition: Factors by (1 - sigma)}
	\end{enumerate}
\end{defn}

\noindent Axiom \ref{Condition: Factors by (1 - sigma)} gives us the following commutative diagram of $\R$-modules:
	\[
		\xymatrix@R=50pt{
		\Domain \ar[rr]^{\Sum} \ar@{->>}[dr]& & \Codomain \\
		& \Domain /(1 - \sigma) \ar@{.>}[ur]_{\widetilde{\Sum}} &
		}
	\]

We write $(\Domain, \Sum)$ or simply $\Sum$ for the summation $(\R, \Codomain, \Domain, \Sum)$ when no confusion results. We write $\Sums \R \Codomain$ for the set of all summations on $\R$ to $\Codomain$.

When $\R = \Codomain = \bbC$ these axioms agree with Hardy's axiom set above \cite[page 6]{Hardy}. Axiom \ref{Condition: AreSeries} (which Hardy takes for granted) guarantees that we are actually summing series rather than (for instance) evaluating definite integrals. Axiom \ref{Condition: IsModuleMorphism} implies that $\Sum$ is linear (Hardy's axioms (A) and (B)). Axioms \ref{Condition: IsModuleMorphism}, \ref{Condition: ExtendsAdd}, and \ref{Condition: Factors by (1 - sigma)} imply that 
\[
	\Sum\parent{a_0 + \dots + a_N + 0 + 0 + \dots} = a_0 + \dots + a_N \in \Codomain,
\]
so a summation extends finitary addition. Axiom \ref{Condition: Factors by (1 - sigma)}, equivalent to Hardy's \ref{HardyC}, requires a summation to be invariant under left and right shifts,  so that the sum of $0 + a_0 + a_1 + \dots$ is the same as the sum of $a_0 + a_1 + \dots$, and one sum is defined if the other is.

Every summation method defined above satisfies axioms \ref{Condition: AreSeries} through \ref{Condition: Factors by (1 - sigma)}, with the exception of exponential and integral Borel summation. However, both $\Sum_b$ and $\Sum_b\prm$ satisfy axioms \ref{Condition: AreSeries}, \ref{Condition: IsModuleMorphism}, \ref{Condition: ExtendsAdd}, and \ref{Hardy subaxiom 1}. We also have the implications 
\[
\Sum_b(X)=x\quad\Rightarrow\quad  \Sum_b\prm(X)=x \quad\Leftrightarrow\quad \Sum_b(\sigma X)=x,
\]
and the first implication is not reversible \cite[Theorems 124, 125, 126]{Hardy}. We conclude that both summations are right-shift-invariant but not, in general, left-shift-invariant (see \cite[Theorem 127]{Hardy} or \cite[Theorem 3.8]{Shawyer-Watson}). 

Given axiom \ref{Hardy subaxiom 1}, left shifting may render a series unsummable, but cannot make it sum to a new value.  Thus, setting
$\Borelsum(X) \coloneqq \Sum_b (\sigma^N X)$
when there exists $N$ such that the right-hand side is defined, we obtain a well-defined summation, called the \defi{Borel summation}. We write $\Domain_B$ for the vector space of Borel summable series. Naturally, we get the same summation if we start with $\Sum\prm_b$ instead of $\Sum_b$. This technique lets us extend any map that satisfies axioms \ref{Condition: AreSeries}, \ref{Condition: IsModuleMorphism}, \ref{Condition: ExtendsAdd}, and \ref{Hardy subaxiom 1} (or, with minor modifications, \ref{Hardy subaxiom 2}) to a summation.

The polynomial $1-\sigma$ internally represents the operation $\Delta$ that takes every series to its series of differences. Its multiplicative inverse 
\[
\Sigma \coloneqq \frac{1}{1 - \sigma} = \sum_{n = 0}^\infty \sigma^n = 1+1+1+1+1+\cdots
\] internally represents the operation $\Sigma$ that takes any series to its series of partial sums. The series $\Sigma$ cannot be in the domain of any summation. Otherwise we would have 
\begin{align*}
	\Sum(1+1+1+\cdots) &= \Sum(1+0+0+\cdots) + \Sum(0+1+1+\cdots) \\
	&= 1 + \Sum(1+1+1+...)\notag
\end{align*}
which is impossible. Thus the domain $\Domain$ is a proper submodule of $\R[[\sigma]]$, and specifying a domain will always be an important part of specifying a summation. 

\begin{definition}
	If $(\Domain, \Sum)$ and $(\Domain\prm, \Sum\prm)$ are summations in $\Sums \R \Codomain$ such that $\Domain \subseteq \Domain\prm$ and $\Sum\prm |_{\Domain} = \Sum$, then we call $\Sum\prm$ an \defi{extension} of $\Sum$. If $\Sum\prm$ extends $\Sum$, we write $\Sum \subseteq \Sum\prm$ or $(\Domain, \Sum) \subseteq (\Domain\prm, \Sum\prm)$. If $\Sum \subseteq \Sum\prm$ and $\Sum \neq \Sum\prm$, we write $\Sum \subset \Sum\prm$.
\end{definition}

As an example, every summation extends $\Add$.

Hardy defines a summation $\Sum$ to be \defi{regular} if ``it sums every convergent series to its ordinary sum'' \cite[page 10]{Hardy}. Thus a summation $\Sum \in \Sums \bbC \bbC$ is regular precisely if it extends $\Sum_c$. For instance, $\Convsum \subset \Sum_{CH} \subset \Sum_A \subset \Eulersumspecial$ (see Example \ref{EulerNotAbel} below), so $\Convsum$, $\Sum_{CH}$, $\Sum_A$, and $\Eulersumspecial$ are all regular summations.

We let $\SumCat$ be the posetal category with summations as objects and and extensions as morphisms. The set $\Sums \R \Codomain$ of summations on $\R$ to $\Codomain$ is inductively ordered by the extension relation; in other words, every ascending chain of summations has a supremum. Note that for every ring $\R$ and $\R$-algebra $\Codomain$, the set $\Sums \R \Codomain$ is a full subcategory of $\SumCat$. In \cite{Hardy}, Hardy writes that $\Sum\prm$ is \defi{stronger than} $\Sum$ when $\Sum\prm \supset \Sum$. In accordance with modern practice, our ``extends'' is the reflexive closure of this, equivalent to Hardy's ``stronger than or equal to''. 

Hardy defines two summations to be \defi{consistent} if they agree on the intersection of their domains \cite[page 65]{Hardy}. 

\begin{definition}
	If $\Sum$ and $\Sum\prm$ have a common extension, then we will call them \defi{compatible}.
\end{definition}

By \cite[Proposition 1.1]{Dawson}, two summations are consistent if and only if they are compatible; such pairs of summations have a largest common restriction $\Sum \cap \Sum\prm$ and a least common extension $\Sum\vee\Sum\prm$. However, we shall see that this equivalence fails when we impose additional constraints on our summations (see for instance Example \ref{ConsistButNotMultComp} below), and so we will maintain the two distinct terms. The following proposition is straightforward.

\begin{prop}\label{Proposition: CanCon}
The following are equivalent for any extension $(\Sum',\Domain')$ of a given summation $\Sum$:
\begin{enumerate}[label=(\roman*)]
  \item\label{Compat} $\Sum'$ is compatible with every extension of $\Sum$;
  \item\label{Consist} $\Sum'$ is consistent with every extension of $\Sum$;
  \item\label{DomDef} $\Sum'$ is the unique extension of $\Sum$ to $\Domain'$.
\end{enumerate}
\end{prop}

\begin{definition}
	If $\Sum\prm$ satisfies the conditions of Proposition \ref{Proposition: CanCon}, we say $\Sum\prm$ is $\Sum$-\defi{canonical}. The unique $\Sum$-canonical summation that extends every $\Sum$-canonical summation is called the \defi{fulfillment of $\Sum$}. If a summation $\Sum$ is its own fulfillment, we say that $\Sum$ is \defi{fulfilled}.
\end{definition}

The fulfillment of a summation always exists because $\Sums \R \Codomain$ is inductively ordered. A series is in the domain of the fulfillment of $\Sum$ if and only if there is a unique value to which it can be summed by a summation compatible with $\Sum$.

We noted already that $\Sum_c \subseteq \Sum_{CH} \subseteq \Sum_A \subseteq \Sum_E$. We now examine the relationship between $\Sum_B$ and $\Sum_E$.

\begin{example}\label{Example: BorelNotAbel}
	Let 
	\[
	G_{-2} \coloneqq \frac{1}{1 + 2 \sigma} = \sum_{n = 0}^\infty (-2)^n \sigma^n = 1 - 2 + 4 - 8 + 32 - 64 + \dots.
	\]
	For any real number $r$ with $\frac{1}{2} < r < 1$, we have $\abs{-2 r} > 1$, and so $\sum_{n = 0}^\infty (-2)^n r^n$ diverges by the ratio test. Thus $G_{-2}$ is  not Abel summable. However,
\[
\Borelsum(G_{-2}) =  \int_0^\infty e^{-t} \sum_{n = 0}^\infty \frac{(-2)^{n}}{n!} t^n dt =  
\int_0^\infty e^{-t} e^{-2t} dt = \frac{1}{3},
\]
so $G_{-2}$ is Borel summable. 
\end{example}

\begin{example}\label{AbelNotBorel}
	The series $A \coloneqq \sum_{n=0}^\infty(-1)^nn\sigma^{n^2}$ is Abel summable but not Borel summable \cite[page 213]{Hardy}. Hardy gives the following Tauberian theorem for (exponential and integral) Borel summability: if 
	\[
	S \coloneqq \sum_{n=0}^\infty s_n \sigma^n \in \Domain_b \ \text{and} \ s_n = O(\sqrt{n}),
	\]
	then $S$ is convergent \cite[Theorem 156]{Hardy}. But for any fixed $N \in \bbN$, we see $s_n = O(\sqrt{n})$ if and only if $s_{n + N} = O(\sqrt{n})$, and $S$ converges if and only if $\sigma^N S$ converges. So the Tauberian theorem applies equally to our shift-invariant Borel summation $\Sum_B$: if 
	\[
	S = \sum_{n=0}^\infty s_n \sigma^n \in \Domain_B \ \text{and} \ s_n = \ O(\sqrt{n}),
	\]
	then $S$ is convergent. By contrapositive, $A$ is not Borel summable, as claimed.
\end{example}  

Thus neither of $\Sum_A$ nor $\Sum_B$ extends the other.
	
\begin{example}\label{EulerNotAbel}
	The series $G_{-2}$ also demonstrates that $\Eulersumspecial \not\subseteq \Abelsum$, but on the other hand, it is immediate that $\Abelsum \subseteq \Eulersumspecial$. 
\end{example}

Let $X = X(\sigma) \in \Boreldomain$. If $X(z)$ converges in a neighborhood of the origin, then by \cite[Theorem 132]{Hardy}, we have $\sigma^N X(\sigma) \in \Eulerdomainspecial$ for some $N$, and moreover 
	\[
	\Borelsum(X) = \Eulersumspecial(\sigma^N X) = \Eulersumspecial(X).
	\]
	But if $X$ is Euler summable, then $X(z)$ must converge in a neighborhood of the origin, and so $z^N X(z)$ converges in a neighborhood of the origin for every $N \geq 0$. Thus $\Boreldomain \cap \Eulerdomainspecial$ contains exactly the Borel-summable series $X$ for which $X(z)$ converges in a neighborhood of the origin, and the summations $\Borelsum$ and $\Eulersumspecial$ are compatible. 

\begin{example}\label{Example: Euler summation does not extend Borel summation}
	Following \cite[page 190]{Hardy}, let us consider the series 
	\[
	Y \coloneqq \sum_{n = 0}^\infty y_n \sigma^n = y_0 + y_1 + y_2 + y_3 + \dots \ \text{with}  \ y_n \coloneqq \sum_{m = 0}^\infty \frac{(-1)^m m^n}{m!} 2^m.
	\] 
	We easily verify that $\Borelsum(Y) = \int_0^\infty e^{-t - 2 e^t} dt$ converges, and so $Y \in \Boreldomain$. But $Y(z)$ does not converge in any neighborhood of the origin, and so $Y \not\in \Eulerdomainspecial$. Thus $\Boreldomain \not\subseteq \Eulerdomainspecial$.
\end{example}

\begin{example}\label{Example: Borel summation does not extend Euler summation}
	Let 
	\[
	Z \coloneqq \frac{1}{1 + 16 \sigma^4} = 1 + 0 + 0 + 0 - 16 + 0 + 0 + 0 + 256 +  \dots.
	\]
	Suppose by way of contradiction that $\Sum_b\prm(\sigma^N Z)$ converges for some $N \geq 0$. Replacing $N$ with $4N$ preserves Borel convergence \cite[Theorem 127]{Hardy}, so we compute
	\begin{align*}
	\Sum_B(Z) &= \Sum_b\prm(\sigma^{4N} Z) \\
	&= \pm 2^{-4N} \int_0^\infty e^{-t} \parent{\sum_{n = 0}^\infty \frac{(-1)^n 2^{4n} \sigma^{4n}}{(4n)!} - \sum_{n = 0}^{N-1} \frac{(-1)^n 2^{4n} \sigma^{4n}}{(4n)!}} dt \\
	&= \pm 2^{-4N} \int_0^\infty e^{t} \cdot \frac{e^{\zeta t} + e^{\zeta^3 t} + e^{\zeta^5 t} + e^{\zeta^7 t}}{4} dt \mp \sum_{n = 0}^{N-1} (-1)^n 2^{4n-4N},
	\end{align*}
	where $\zeta \coloneqq \exp\parent{\i \pi / 4}$. This integral diverges, and so $Z \not\in \Boreldomain$. On the other hand, $Z(z)$ converges to $\frac{1}{1 + 16 z^4}$ for $\abs{z} < \frac{1}{2}$, and this identity provides an analytic continuation of $Z(z)$ except at the points $\frac{\zeta}{2},$ $\frac{\zeta^3}{2},$ $\frac{\zeta^5}{2},$ and $\frac{\zeta^7}{2}$. In particular, we have an analytic continuation of $Z(z)$ in a neighborhood of the interval $[0,1)$, and
	\[
	\Eulersumspecial (Z) = \frac{1}{17},
	\]
	so $Z \in \Eulerdomainspecial$. Thus $\Eulerdomainspecial \not\subseteq \Boreldomain$
\end{example}

Taken together, Examples \ref{Example: Euler summation does not extend Borel summation} and \ref{Example: Borel summation does not extend Euler summation} show that the summations $\Borelsum$ and $\Eulersumspecial$ are compatible but not directly comparable.

In this paper, we study techniques for systematically and canonically extending summations.

\begin{definition}
A \defi{functor of summations} is a functor from the category $\SumCat$ to (but not necessarily onto) itself. If $\calG$ is another functor of summations, we say that $\calF$ \defi{subsumes} $\calG$ if 
\[
	\calF = \calF \circ \calG = \calG \circ \calF.
\] 
We say $\calF$ is \defi{idempotent} if $\calF$ subsumes itself, so that $\calF = \calF \circ \calF$.
\end{definition}

For $\calF$ a functor of summations, we abuse notation and write $\parent{\calF \Domain, \calF \Sum}$ for $\calF(\Domain, \Sum)$. If $\calF$ subsumes $\calG$, then $\calG$ fixes the image of $\calF$, and for any summation $\Sum$, the summation $\calF \Sum$ extends the summation $\calG \Sum$. The following proposition unpacks the definition of a functor of summations.

\begin{prop}\label{Proposition: functor of summations}
	Let $\calF$ be a rule which assigns a unique summation $\calF \Sum$ to every summation $\Sum$. The following are equivalent:
	\begin{enumerate}[label=(\roman*)]
		\item The rule $\calF$ is a functor of summations.
		\item If $\Sum \subseteq \Sum\prm$, then $\calF \Sum \subseteq \calF \Sum\prm$.
		\item If $\Sum$ and $\Sum\prm$ are compatible, then $\calF \Sum$ and $\calF \Sum\prm$ are compatible.
	\end{enumerate}
\end{prop}

Thus, a functor of summations is simply a rule that preserves compatibility and extensions.

\begin{definition}
	We say a functor of summations $\calF$ is an \defi{extension map (of summations)} if $\calF \Sum \supseteq \Sum$ for every summation $\Sum$. If $\calF$ is an extension map and $\Sum$ is a summation such that $\Sum = \calF \Sum$, we say that $\Sum$ is \defi{$\calF$-closed}.
\end{definition}

By Proposition \ref{Proposition: functor of summations}, every extension map $\calF$ is ``canonical'' in the sense that $\calF \Sum$ is $\Sum$-canonical for every $\Sum$. In particular, if $\calF$ is an extension map and $\Sum \in \Sums \R \Codomain$, then $\calF \Sum \in \Sums \R \Codomain$.

The remainder of this paper is organized as follows. In Section \ref{Section: Telescopic Extensions}, we renovate the definition of telescopic summations from \cite{Dawson} for our present context, and prove that telescopic extension is an extension map. In Section \ref{Section: Multiplicative Summations}, we discuss and defend a stronger axiom schema for summations, namely multiplicativity, and show that it is compatible with telescopic extensions. In Section \ref{Section: Weakly Multiplicative Summations}, we construct the minimal multiplicative extension of a weakly multiplicative summation (such as $\Convsum$), and relate this construction to the telescopic extension of $\Sum$. In Section \ref{Section: Rational Extensions}, we fulfill the promise made in \cite{Dawson} to study N{\o}rlund means relative to series with infinite support by defining the rational extension of a multipicative summation; the rational extension multiplicatively subsumes telescopic extension. Example \ref{RatNotTel} will show that the rational extension of $\Convsum$ is larger than the telescopic extension of $\Convsum$.

\section*{Corrections}

The first author would like to point out that Proposition 2.5 of \cite{Dawson} is in error. If a summation $(\Domain, \Sum)$ is not multiplicative, $\Domain$ is generally not a ring, nor are $\T \Domain$ (see Example \ref{TelNotMult}), $\T \Domain \cup I_\T(\Sum)$, or $\T \Domain \cup \Psi_T(\Domain)$. Moreover, as $0 \not\in I_\T(\Sum)$ and $0 \not\in \Psi_T(\Domain)$, even if these modules are rings, $I_\T(\Sum)$ and $\Psi_\T(\Sum)$ cannot be their ideals. In our upcoming paper \cite{Dawson-Molnar-3} we will prove a corrected version of the claim.

Furthermore, while Theorem 4.2 of that paper, and the part of Theorem 4.4 that concerns $\Sum_c$, are correct, the corresponding statements about absolutely convergent series (Theorem 4.3, part of Theorem 4.4, and Corollary 4.4.1) are in error. The problem in each case is that the convergence to 0 of $X$ as a sequence need not be ``absolute,'' in the sense of the series $\sum(\sigma-1)X$ converging absolutely. The alternating harmonic series, which is telescopable over $\Sum_a$, is a counterexample to Corollary 4.4.1.

\section{Telescopic Extensions Revisited}\label{Section: Telescopic Extensions}

Telescoping convergent series is an ancient technique, dating back to Archimedes \cite[page 233]{Heath}, that we still teach our modern calculus students \cite[pages 462-464]{Calculus}. Euler was comfortable telescoping divergent series as well as convergent ones \cite[pages 205-209]{Euler}, and Hardy was aware of how the method interacted with his axioms for summations \cite[page 6]{Hardy}. Telescoping is also implicit in the study of linear recurrences \cite{Myerson-Poorten}.

In 1997, the first author introduced the notion of \defi{telescopically extending a summation} $\Sum$ from an integral domain $\R$ to its field of fractions. This method formalizes the familiar trick of telescoping a divergent series, by using the linearity and translation-invariance of $\Sum$ to determine what values a series might be summed to. Telescopic extensions were independently rediscovered (in a slightly different form) by Madhav Nori in 2012 \cite{Nori}. In this section, we generalize the construction in \cite{Dawson} to allow $\R$ to be an arbitrary commutative unital ring, and $\Codomain$ to be an arbitrary commutative unital $\R$-algebra.

Recall that $x \in \Codomain$ is \defi{regular} if $x$ is not a zerodivisor. Here and throughout the paper, we write $\Reg \Codomain$ for the set of regular elements of $\Codomain$.

\begin{defn}\label{Definition: Telescopic Extension}
	For any summation $(\Domain, \Sum) \in \Sums \R \Codomain$, the \defi{telescopic extension} $(\T \Domain, \T \Sum)$ of $(\Domain, \Sum)$ is defined as follows. For $X \in \R[[\sigma]]$, we say $X \in \T\Domain$ if there exists $A \in \Domain$, $F \in \R[\sigma]$, $f \in \Reg \Codomain$, and $x \in \Codomain$ such that $A = FX$, $\Sum(F) = f$, and $\Sum(A) = f x$. We define
	\begin{align*}
		\T \Sum &: \T \Domain \to \Codomain, \\
		\T \Sum &: X \mapsto x \ \text{if} \ X \ \text{is as above.}
	\end{align*}
\end{defn}

The following theorem extends facts proven in \cite{Dawson} to our current framework (see also Theorem \ref{Theorem: Rational Extension} below).

\begin{theorem}\label{Theorem: Telescopic Extension}
	The functor of summations 
	\[
		\T : \SumCat \to \SumCat
	\]
	is an idempotent extension map.
\end{theorem}

\begin{proof}
	Our argument mirrors the construction of the field of fractions of an integral domain. Let $(\Domain, \Sum) \in \Sums \R \Codomain$ be arbitrary. We first prove that $\T \Sum$ is a well-defined function on $\T \Domain$. Suppose $X \in \T \Domain$, and let $A \in \Domain$ and $F \in \R[\sigma]$ be as in the definition of telescopic extensions, so that $A = F X$, $\Sum \parent F = f \in \Reg \Codomain$, and $\Sum \parent A = f x$. Now suppose that we have $A\prm \in \Domain$ and $F\prm \in \R[\sigma]$ satisfying the same array of properties, so that $A\prm = F\prm X$, $\Sum \parent{F\prm} = f\prm \in \Reg \Codomain$, and $\Sum \parent{A\prm} = f\prm x\prm$. Then $F A\prm = F F\prm X = F\prm A \in \Domain$, and we see 
	\[
	f f\prm x\prm = \Sum\parent{F A\prm} = \Sum \parent{F F\prm X} = \Sum\parent{F\prm A} = f\prm f x.
	\]
	But the regularity of $f$ and $f\prm$ implies that $f f\prm$ is regular as well, and so $x = x\prm$. Then $\T\Sum$ is well-defined as a function. Arguments of a similar style verify that $\T\Sum$ validates our summation axioms, and by construction $\Domain \subseteq \T\Domain$. Thus $\T$ is an extension map.
	
	It remains to show that $\T\Sum = \T\T\Sum$. So suppose that $X \in \T\T\Domain$. Then there exists $F \in \R[\sigma]$ and $A \in \T\Domain$ such that $\Sum \parent F = f \in \Reg \Codomain$, $\T\Sum \parent A = f x$ for some $x \in \Codomain$, and $A = F X$. But now as $A \in \T\Domain$, there exists $G \in \R[\sigma]$ and $B \in \Domain$ such that $\Sum \parent G = g \in \Reg \Codomain$, $\Sum\parent{B} = g y$ for some $y \in \Codomain$, and $B = G A$. Then 
	\[
	F G X = G A = B,
	\]
	and
	\[
	\Sum\parent{B} = \Sum\parent{G A} = g f x.
	\]
	But $gf = \Sum\parent{GF}$ is regular, so $X \in \T\Domain$, and $\T \Sum = \T\T\Sum$ as claimed.
\end{proof}

If $\Sum = \T \Sum$, so that $\Sum$ is $\T$-closed, we say $\Sum$ is \defi{telescopically closed}.

\begin{obsrem}
	The collection $\TSumCat$ of telescopically closed summations forms a full subcategory of $\SumCat$. Observe that $\T$ is left adjoint to the forgetful functor from $\TSumCat$ to $\SumCat$. 
\end{obsrem}

Proposition \ref{Corollary: T is often the fulfillment} below follows almost verbatim from the proof of \cite[Proposition 2.3]{Dawson}. The only new subtlety is the observation that if $F X = A \in \Domain$ with $\Sum \parent{F} = f \in \Reg \Codomain$ but there exists no $x \in \Codomain$ with $\Sum \parent{A} = fx$, then $X$ is not summable by any extension of $\Sum$.

\begin{prop}\label{Corollary: T is often the fulfillment}
	If $\Sum \in \Sums \R \Codomain$ and $\Codomain$ is an integral domain, then $\T\Sum$ is the fulfillment of $\Sum$.
\end{prop}

This proposition is easily adapted to show that if $\Codomain \simeq \Codomain_1 \times \dots \times \Codomain_m$ where $\Codomain_j$ is an integral domain for each $j$, then $\T\Sum$ is the fulfillment of $\Sum$. However, the following example shows $\T \Sum$ need not be the fulfillment of $\Sum$ when $\Codomain$ is not of a direct product of integral domains.

\begin{example}
	Let $\R \coloneqq \bbC[\epsilon_1, \epsilon_2, \dots]/(\epsilon_1^2, \epsilon_2^2, \dots)$, let $\Codomain \coloneqq \R$, and let $X$ be transcendental over $\bbC[\sigma]$. Let
	\[
	\Domain \coloneqq \R[\sigma] \oplus \R[\sigma] \cdot \epsilon_1 X \oplus \R[\sigma] \cdot \epsilon_2 X \oplus \dots,
	\]
	and let $\Sum : \Domain \to \Codomain$ be given by
	\[
	\Sum : F_0 + \epsilon_1 X F_1 + \epsilon_2 X F_2 + \dots \mapsto \Add(F_0) + \epsilon_1 \Add(F_1) + \epsilon_2 \Add(F_2) + \dots;
	\]
	note that this is a finite sum. Then $X \not\in \T\Domain$, since $F X \in \Domain$ implies $F$ is nilpotent, but $X \mapsto 1$ in the fulfillment of $\Sum$.
\end{example}

In general, we face an interesting question.

\begin{quest}
	Let $\Sum \in \Sums \R \Codomain$ be arbitrary. If $\Codomain$ is not an integral domain, then what is the fulfillment of $\Sum$?
\end{quest}

We close this section with two examples that highlight the delicacy of $\T$ and its dependence on the structure of $\Codomain$.

\begin{example}\label{Example: restricting codomain causes issues}
	Let $\R = \bbZ$, let $\Codomain_1 \coloneqq \R = \bbZ$, and let $\Codomain_2 \coloneqq \bbQ$. We observe that there is a natural inclusion $\iota : \Codomain_1 \into \Codomain_2$, so every summation $\Sum \in \Sums \R {\Codomain_1}$ has a natural image $\iota \Sum \in \Sums \R {\Codomain_2}$. For $j \in \set{1,2}$, let $\Add_j : \R[\sigma] \to \Codomain_j$ be the summation on finitely supported series, so that $\Add_2 = \iota \Add_1$. 
	
	We claim $\T \Add_2 \neq \iota \T \Add_1$. Indeed, the series 
	\[
	G_{-1} \coloneqq \frac{1}{1 + \sigma} = 1 - 1 + 1 - 1 + \dots
	\]
	sums to $\frac{1}{2}$ under $\T\Add_2$, but cannot be summed by $\T\Add_1$, since $\frac{1}{2} \not\in \Codomain_1$. Thus extending $\Codomain$ may extend the domain of $\T \Sum$.
\end{example}

\begin{example}\label{Example: extending codomain causes issues}
	Let $\R \coloneqq \bbQ[t]$, let $\Codomain_1 \coloneqq \R = \bbQ[t]$ and let $\Codomain_2 \coloneqq \bbQ[t,u]/(tu)$. Again, there is a natural inclusion $\iota : \Codomain_1 \into \Codomain_2$, so every summation $\Sum \in \Sums \R {\Codomain_1}$ has a natural image $\iota \Sum \in \Sums \R {\Codomain_2}$. For $j \in \set{1,2}$, let $\Add_j : \R[\sigma] \to \Codomain_j$ be the summation on finitely supported series, so that $\Add_2 = \iota \Add_1$. 
	
	We claim $\T \Add_2 \neq \iota \T \Add_1$. Indeed, the series 
	\[
	\frac{t}{1 - \sigma + t\sigma^2} = t + t + t (1 - t) + t (1 - 2 t) + t (1 - 3 t + t^2) +  t (1 - 4 t + 3 t^2) + \dots
	\]
	sums to $1$ under $\T\Add_1$, but cannot be summed by $\T\Add_2$, since $\Add\parent{1 - \sigma + t \sigma^2} = t$ is regular in $\Codomain_1$ but is a zerodivisor in $\Codomain_2$. Thus extending $\Codomain$ may restrict the domain of $\T \Sum$.
\end{example}

\section{Multiplicative Summations}\label{Section: Multiplicative Summations}

While linearity and shift-invariance have been considered to be the properties that define a summation, in fact most standard summation methods also preserve products.

\begin{definition}
	We call a summation \defi{premultiplicative} if, for all $n$, we have
	\begin{equation*}
		\prod_{j=1}^n\Sum \parent{A_j}= \Sum\parent{\prod_{j=1}^n A_j}
	\end{equation*}
	whenever both sides are defined, and \defi{multiplicative} if moreover the existence of the left side implies the existence of the right. 
\end{definition}

\begin{obsrem}For a summation to be multiplicative it is sufficient that $\Sum(A)\Sum(B) = \Sum(AB)$ whenever the left side exists. Premultiplicativity cannot be reduced to the pairwise case; it is possible that the series $A,B,C,AB,BC$, and $CA$ are all summable, but $ABC$ is not (see Example \ref{nPremultButNotPremult}). 
\end{obsrem}

The tuple $(\R, \Codomain, \Domain, \Sum)$ is a multiplicative summation on $\R$ to $\Codomain$ if it satisfies axioms \ref{Condition: AreSeries}, \ref{Condition: ExtendsAdd}, \ref{Condition: Factors by (1 - sigma)} from Section \ref{Section: Introduction}, and the following strengthened version of axiom \ref{Condition: IsModuleMorphism}:
\begin{enumerate}[label=(\Roman*$\prm$)]
	\setcounter{enumi}{1}
	\item The set $\Domain$ is an $\R$-algebra, and the map $\Sum$ is an $\R$-algebra homomorphism; \label{Condition: IsAlgebraMorphism}
\end{enumerate}

\noindent In this context, axiom \ref{Condition: Factors by (1 - sigma)} gives us the following commutative diagram of $\R$-algebras:
	\[
		\xymatrix@R=50pt{
		\Domain \ar[rr]^{\Sum} \ar@{->>}[dr]& & \Codomain \\
		& \Domain /(1 - \sigma) \ar@{.>}[ur]_{\widetilde{\Sum}} &
		}
	\]

We write $\PSums \R \Codomain $ and $\MSums \R \Codomain$ respectively for the sets of premultiplicative summations and multiplicative summations on $\R$ to $\Codomain$.
We take the category $\pMSumCat$ to be the full subcategory of $\SumCat$ with objects the premultiplicative summations, and we likewise take the category $\MSumCat$ to be the full subcategory of $\SumCat$ with objects the multiplicative summations. For further discussion of multiplicative summations in the cases $\R = \Codomain = \bbR$ and $\R = \Codomain = \bbC$, the reader is referred to \cite[Chapter 10]{Hardy}.

\begin{definition}
	If a summation has an extension which is multiplicative, we call it \defi{weakly multiplicative}.
\end{definition}

We write $\WSums \R \Codomain$ for the set of weakly multiplicative summations on $\R$ to $\Codomain$. We take $\wMSumCat$ to be the full subcategory of $\SumCat$ with objects the weakly multiplicative summations. 

Clearly 
\[
	\MSums \R \Codomain \subseteq \WSums \R \Codomain \subseteq \PSums \R \Codomain \subseteq \Sums \R \Codomain,
\]
and likewise
\[
	\MSumCat \subseteq \wMSumCat \subseteq \pMSumCat \subseteq \SumCat
\]
as categories. 

Analysis abounds with examples of (weakly) multiplicative summations. We recall a few examples here.

\begin{itemize}
  	\item The summation $(\R[\sigma], \Add) \in \Sums \R \Codomain$, which is defined only on finitely-supported series, is multiplicative.

	\item While a product of two convergent series $A$ and $B$ need not converge (consider the product of $(1-1/\sqrt{2}+1/\sqrt{3}-\cdots)$ with itself), Abel's product theorem \cite[Theorem 162]{Hardy} tells us the classical summation $(\Domain_c, \Convsum) \in \Sums \bbC \bbC$ on convergent series is premultiplicative.
	
	\item By Cauchy's product theorem \cite[Theorem 160]{Hardy}, the restriction $(\Domain_a, \Abssum) \in \Sums \bbC \bbC$ of $(\Domain_c, \Convsum)$ to absolutely convergent series is multiplicative.
	
	\item In $\bbQ_p$, a series is convergent if and only if it is absolutely convergent \cite[page 76]{Koblitz}. Thus, essentially by Cauchy's product theorem, the summation $(\Sum_p, \Domain_p) \in \Sums {\bbQ_p} {\bbQ_p}$ on $p$-adically convergent series is multiplicative.	
	
	\item  By Ces\`aro's product theorem \cite[Theorems 163, 164]{Hardy}), the Ces\`aro-H\"older summation $(\CHdomain, \CHsum) \in \Sums \bbC \bbC$ is a multiplicative summation that extends $\Convsum$. We conclude that $\Convsum$ is weakly multiplicative, not just premultiplicative.
	
	\item The Abel summation $(\Abeldomain, \Abelsum) \in \Sums \bbC \bbC$ is another multiplicative summation. It extends $\CHsum$ \cite[page 228]{Hardy}, and \textit{a fortiori} extends $\Convsum$.
	
	\item The Borel summation $(\Boreldomain, \Borelsum) \in \Sums \bbC \bbC$ is a weakly multiplicative summation that extends $\Convsum$ (see \cite[Theorems 122, 189]{Hardy}, as well as Doetsch's Theorem \cite[page 246]{Hardy}).
	
	\item For every connected set $\Omega$ containing $0$ with $1$ as a limit point, the Euler sum $(\Eulerdomain \Omega, \Eulersum \Omega)$ is a multiplicative summation that extends $\Abelsum$, and \textit{a fortiori} extends $\Convsum$. This follows by the uniqueness of analytic continuations, since the product of analytic functions is analytic. 
\end{itemize}

The definition of consistency carries over unchanged to the multiplicative case; so the largest common multiplicative restriction of two or more multiplicative restrictions is still their intersection. 

\begin{definition}
	If a set of multiplicative summations have a common multiplicative extension, we will call them \defi{multiplicatively compatible}.
\end{definition}
	
The next example shows that multiplicative compatibility, even in its pairwise form (see Example \ref{nMultCompButNotMultComp} below), is strictly stronger than consistency. 

\begin{example}\label{ConsistButNotMultComp} 
Let $(\Domain,\Sum) \in \MSums \R \Codomain$ be a multiplicative summation, and $V$ an invertible series that is transcendental with respect to $\Domain$. We may define a summation $\Sum' : \Domain[V] \to \Codomain$ so that 
\[
\Sum':\sum_{j=0}^nA_jV^j \mapsto \Sum(A_0),
\]
and a second summation $\Sum'' : \Domain[V^{-1}] \to \Codomain$ so that 
\[
\Sum'':\sum_{j=0}^nA_jV^{-j} \mapsto \Sum(A_0).
\]
Clearly $\Domain[V]\cap\Domain[V^{-1}] = \Domain$, so the summations $\Sum'$ and $\Sum''$ are consistent (and thus compatible). However, there can be no common multiplicative extension, as $\Sum'(V)=\Sum''(V^{-1})= 0$, but $\Sum(VV^{-1}) = \Sum(1) = 1$.
\end{example}

\begin{obsrem} In the construction of counterexamples, it is often convenient to extend a summation $(\Domain,\Sum)$ by a series $V$ that is transcendental with respect to $\Domain$. In a multiplicative extension, $\Sum'(V)$ can be chosen freely; in a nonmultiplicative extension, the values of $V$ and its powers are linearly independent. Series transcendental with respect to a given ring of series can be constructed, or shown to exist, in several ways.
\begin{itemize}
  \item Over $\bbQ[\sigma]$, any convergent series in $\bbQ[[\sigma]]$ with transcendental sum is transcendental.
  \item For any finite or countable ring $R$, $R[\sigma]$ is countable and $R[[\sigma]]$ is uncountable; thus there are uncountably many series in $R[[\sigma]]$ algebraically independent over $R[\sigma]$ \cite[Example 1.1]{Dawson-Molnar-1}.
  \item For any ring $R$, a Liouville-type argument shows that $\sum_{n=0}^\infty \sigma^{n!}$ is transcendental over $R[\sigma]$.       
  \item If $\Domain$ is the subring of $\bbC[[\sigma]]$ consisting of series that exhibit at most exponential growth, any series exhibiting superexponential growth is transcendental over $\Domain$ \cite[Corollary 1.4]{Dawson-Molnar-1}. Iterated application of \cite[Theorem 3.2]{Dawson-Molnar-1} yields a sequence $X_1, X_2, \dots$ of series algebraically independent over the ring of series exhibiting at most exponential growth.  
\end{itemize}
\end{obsrem} 

When $\Sum$ and $\Sum'$ are multiplicatively compatible, we can define the least common multiplicative extension $\Sum \triangledown \Sum'$. It is an  extension, generally proper, of $\Sum \vee \Sum'$.

An example similar to Example \ref{ConsistButNotMultComp} shows that not every summation is premultiplicative.

\begin{example}\label{NotPremult}
With $(\Domain,\Sum)$ and $V$ as in Example \ref{ConsistButNotMultComp}, let 
\[
\Domain\prm \coloneqq \Domain \oplus \Domain V \oplus \Domain V\inv,
\]
and define $\Sum\prm : \Domain\prm \to \Codomain$ by
\[
\Sum\prm: A+BV+CV^{-1}  \mapsto \Sum(A).
\]
Then $\Sum\prm(V)=\Sum\prm(V^{-1})= 0$ but $\Sum\prm(VV^{-1}) = \Sum\prm(1) = 1$.
\end{example}

The definition of a premultiplicative summation involves products of all orders; this is not implied by the corresponding statement for pairwise products (or indeed by products of any other order.) The next example exhibits a summation such that $\Sum(A)\Sum(B) = \Sum(AB)$ whenever both sides are defined, but where this can fail for products of three series. 

\begin{example}\label{nPremultButNotPremult} 
Let $(\Domain,\Sum) \in \MSums \R \Codomain$ be a multiplicative summation, let $V_1$, $V_2$ be invertible series that are algebraically independent over $\Domain$, and let $V_3 \coloneqq (V_1V_2)^{-1}$. Let 
\[
\Domain' \coloneqq \Domain \oplus \Domain \cdot V_1 \oplus \Domain \cdot V_2 \oplus \Domain \cdot V_3,
\]
and define $\Sum\prm : \Domain\prm \to \Codomain$ by
\[
\Sum': A_0+A_1V_1 + A_2V_2 + A_3V_3 \mapsto \Sum(A_0)
\]
The only products of elements of $\Domain'$ that are themselves in $\Domain'$ have the form 
\[
A(A_0+A_1V_1 + A_2V_2 + A_3V_3) \ \text{for} \ A,A_0,A_1,A_2,A_3 \in \Domain.
\]
In particular, $V_1V_2$, $V_2V_3$, and $V_3V_1$ are not in $\Domain'$, so $(\Domain',\Sum')$ is trivially 2-premultiplicative. However, $\Sum\prm$ is not premultiplicative, as 
\[
\Sum'(V_1V_2V_3)=1 \neq \Sum'(V_1)\Sum'(V_2)\Sum'(V_3).
\]
\end{example}

Example \ref{nPremultButNotPremult} can be easily extended to yield a summation that is $k$-premultiplicative but not $(k+1)$-premultiplicative for any $k \geq 2$.

Our next example shows that weak multiplicativity is strictly stronger than premultiplicativity.

\begin{example}\label{PremultButNotWeakMult}
 Let $V_1,$ $V_2$ be algebraically independent over $\R[\sigma]$, let $V_3 \coloneqq (V_1V_2)^2 - V_1V_2$, and let $\Sum$ be the minimal extension of $\Add$ taking $V_1,$ $V_2$, and $V_3$ to $1$. The domain $\Domain$ of $\Sum$ is
\[
\Domain = \R[\sigma] \oplus \R[\sigma] \cdot V_1 \oplus \R[\sigma] \cdot V_2 \oplus \R[\sigma] \cdot V_3,
\]
and $\Sum$ is trivially premultiplicative. But $\Sum$ has no multiplicative extension, because if $V_1,V_2\mapsto 1$, then in a multiplicative summation, we must have $V_3 \mapsto 0$.
\end{example}

The following example exhibits three multiplicative summations that are pairwise multiplicatively compatible but not multiplicatively compatible overall. Again, there is an obvious generalization to the $n$-fold case.

\begin{example}\label{nMultCompButNotMultComp} 
Let $V_1,V_2,V_3$ be as in Example \ref{nPremultButNotPremult}, and let 
\[
\Domain_i \coloneqq \Domain \oplus \Domain \cdot V_i
\]
for $i=1,2,3$, and 
\[
\Sum_i: A_0 + A_i V_i \mapsto \Sum(A_0).
\]
Then $(\Domain_1,\Sum_1)$ and $(\Domain_2,\Sum_2)$ have a common multiplicative extension such that
$$\Sum_{12}:\sum_{j,k} A_{jk}V_1^j V_2^k \mapsto \Sum(A_{0}).$$
Extensions $\Sum_{13}$ and $\Sum_{23}$ can be defined analogously; but any common extension of $(\Domain_1,\Sum_1)$, $(\Domain_2,\Sum_2)$, and $(\Domain_3,\Sum_3)$  would map $V_1V_2V_3 = 1$ to $0$, which is impossible.
\end{example} 

\begin{definition}
	If $\Sum\prm$ is multiplicatively compatible with every multiplicative extension of $\Sum$, we say $\Sum\prm$ is \defi{multiplicatively $\Sum$-canonical}. The unique multiplicatively $\Sum$-canonical summation that extends every multiplicatively $\Sum$-canonical summation is called the \defi{multiplicative fulfillment of $\Sum$}. If a summation $\Sum$ is its own multiplicative fulfillment, we say that $\Sum$ is \defi{multiplicatively fulfilled}.
\end{definition}

This definition is adequate despite Example \ref{nMultCompButNotMultComp}.  Certainly if $\{\Sum_1,\dots,\Sum_n\}$ are not jointly multiplicatively compatible, there is no point asking whether $\Sum_0$ is jointly multiplicatively compatible with all of them. On the other hand, if $\{\Sum_1,\dots,\Sum_n\}$ are jointly multiplicatively compatible, and 
$\Sum'$ is multiplicatively canonical, then $\Sum'$ is by hypothesis multiplicatively compatible with $\MJoin_{j=0}^n \Sum_j$.

The multiplicative fulfillment of $\Sum$ is the largest ``choice-free'' multiplicative extension of $\Sum$. The fulfillment of $\Sum$ in $\SumCat$ is defined ``pointwise'': that is, it consists of precisely the series that can be summed to a unique value by a $\Sum$-compatible summation. This is not true for multiplicative fulfillments, as Example \ref{nMultCompButNotMultComp} shows.

The paragraph before Examples \ref{Example: Euler summation does not extend Borel summation} and \ref{Example: Borel summation does not extend Euler summation} tells us that Borel summation and Euler summation are compatible, but with our new definitions in hand, we may pose a stronger, subtler question.

\begin{quest}
	Are the summations $\Borelsum$ and $\Eulersumspecial$ multiplicatively compatible?
\end{quest}

\begin{definition}
	We say that a functor $\calF$ of (weakly) multiplicative summations \defi{preserves multiplicative compatibility} if $\calF \Sum$ and $\calF \Sum\prm$ are multiplicatively compatible whenever $\Sum$ and $\Sum'$ are multiplicatively compatible; such a map also preserves multiplicative extensions. The functor $\calF$ is \defi{multiplicatively canonical} if $\calF \Sum$ is multiplicatively $\Sum$-canonical for all multiplicative $\Sum$. We carry over the notions of extension map, idempotence, and functors subsuming functors directly from $\SumCat$ to functors on subcategories of $\SumCat$.
\end{definition}

Our next proposition proves that telescopic extension respects multiplicativity, and indeed multiplicative compatibility.

\begin{proposition}\label{Proposition: Telescopic extensions preserve multiplication}
	If $\Sum$ is a multiplicative (respectively weakly multiplicative, premultiplicative)  summation, so is $\T\Sum$. Thus
	\begin{eqnarray*}
		\T &:& \MSumCat \to \MSumCat,\\
		\T &:& \wMSumCat \to \wMSumCat, \mbox{ and }\\
		\T &:& \pMSumCat \to \pMSumCat
	\end{eqnarray*}
	are extension maps.
\end{proposition}

\begin{proof}
	Let $(\Domain, \Sum) \in \Sums \R \Codomain$ be a premultiplicative summation, and suppose that $X_1, X_2,$ and $X_3$ are series in $\T \Domain$. Then there exist series $A_1, A_2, A_3, F_1, F_2, F_3 \in \Domain$, and constants $f_1, f_2, f_3, x_1, x_2, x_3 \in \Codomain$ such that for each $j \in \set{1, 2, 3}$,
	\begin{itemize}
		\item We have $A_j = F_j X_j$;
		\item We have $\Sum (A_j) = f_j x_j$ and $\Sum (F_j) = f_j$;
		\item We have $F_j \in R[\sigma]$ and $f_j \in \Reg \Codomain$. 
	\end{itemize}
	Recall that for any $F$ finitely supported and $X \in \Domain$, we have $\Sum \parent{F X} = \Sum \parent F \Sum \parent X$. Now suppose that $X_1 X_2 = X_3$, and consider the series 
	\[
		F_3 A_1 A_2 = F_1 F_2 F_3 X_1 X_2 = F_1 F_2 F_3 X_3 = F_1 F_2 A_3 \in \Domain.
	\]
	As $F_3 A_1, A_2 \in \Domain$, and $F_3 A_1 A_2 \in \Domain$, we have by premultiplicativity that
	\[
		f_1 f_2 f_3 x_3 = \Sum\parent{F_1 F_2 A_3} = \Sum\parent{F_3 A_1 A_2} = \Sum\parent{F_3 A_1} \Sum\parent{A_2} = f_3 f_1 x_1 f_2 x_2 = f_1 f_2 f_3 x_1 x_2.
	\]
	Then as $f_1 f_2 f_3$ is regular, we see $x_3 = x_1 x_2$ and so $\T\Sum\parent{X_1 X_2} = \T\Sum\parent{X_1} \T\Sum\parent{X_2}$. If moreover $\Sum$ is multiplicative, then we know $X_1 X_2 \in \T\Domain$ whenever $X_1, X_2 \in \T\Domain$, and the same argument goes through unimpeded. Finally, as $\T$ preserves extensions, if $\Sum$ is weakly multiplicative with multiplicative extension $\Sum\prm$, then $\T\Sum$ is also weakly multiplicative with  multiplicative extension $\T\Sum'$. Thus $\T$ takes objects in $\MSumCat$ to objects in $\MSumCat$, objects in $\wMSumCat$ to objects in $\wMSumCat$, and objects in $\pMSumCat$ to objects in $\pMSumCat$.
	
	Now as $\T : \SumCat \to \SumCat$ preserves extensions, it induces defines extension maps 
	\begin{eqnarray*}
		\T &:& \MSumCat \to \MSumCat,\\
		\T &:& \wMSumCat \to \wMSumCat, \mbox{ and }\\
		\T &:& \pMSumCat \to \pMSumCat
	\end{eqnarray*}
	as desired.
\end{proof}

\begin{corollary}\label{Corollary: Telescopic extensions are multiplicatively canonical}
	If $\Sum$ is weakly multiplicative, then $\T \Sum$ is a multiplicatively canonical extension of $\Sum$.
\end{corollary}

\section{Weakly Multiplicative Summations}\label{Section: Weakly Multiplicative Summations}

We now have a firm understanding of telescopic extensions, and a repository of examples to work with. We have demonstrated that $\wMSumCat$ is properly contained in $\pMSumCat$, albeit only with artificial examples. Let us turn our attention to the relationship between weakly multiplicative summations and multiplicative summations.

\begin{defn}\label{Definition: Multiplicative Extension}
	For any weakly multiplicative summation $(\Domain, \Sum) \in \WSums \R \Codomain$, the \defi{multiplicative extension} $(\M \Domain, \M\Sum)$ of $(\Domain, \Sum)$ is defined as follows. For $X \in \R[[\sigma]]$, we say $X \in \M\Domain$ if $X = \sum_{i = 1}^\ell \prod_{j = 1}^{k_i} X_{i,j}$ for some $\ell, k_i \in \bbN$ and $X_{i,j} \in \Domain$. We define
	\begin{align*}
		\M\Sum &: \M\Domain \to \Codomain, \\
		\M\Sum &: X \mapsto  \sum_{i = 1}^\ell \prod_{j = 1}^{k_i} \Sum \parent{X_{i,j}} \ \text{if} \ X \ \text{is as above.}
	\end{align*}
\end{defn}

\begin{theorem}\label{Theorem: Multiplicative Extension}
	The functor of summations
	\[
		\M : \wMSumCat \to \MSumCat
	\]
	is an idempotent extension map which sends every weakly multiplicative summation to its unique minimal multiplicative extension.
\end{theorem}

\begin{proof}
	Let $(\Domain, \Sum) \in \WSums \R \Codomain$, and let $(\Domain\prm, \Sum\prm)$ be any multiplicative extension of $(\Domain, \Sum)$. As $\M \Domain$ is the smallest ring containing $\Domain$ by construction, and $\M \Sum = \Sum\prm |_{\M \Domain}$ by observation, we see that $\M\Sum$ is a well-defined extension map. Furthermore, $\M\Sum$ is a minimal multiplicative summation extending $\Sum$, by the minimality of $\M\Domain$. Finally, since $\M\Sum$ agrees with $\Sum\prm$ on $\M\Domain$ and $\Sum\prm$ was an arbitrary multiplicative extension of $\Sum$, we see that $\M\Sum$ is the unique minimal extension. The remaining claims follow immediately.
\end{proof}

\begin{obsrem} 
	Observe $\M$ is left adjoint to the forgetful functor from $\MSumCat$ to $\wMSumCat$.
\end{obsrem}

The following proposition gives us a concrete, albeit straightforward, criterion for recognizing weakly multiplicative summations.

\begin{prop}\label{Proposition: Equivalent Conditions for Being Weakly Multiplicative}
	Let $(\Domain, \Sum) \in \Sums \R \Codomain$ be arbitrary. The following are equivalent:
	\begin{enumerate}[label=(\roman*)]
		\item The summation $(\Domain, \Sum)$ is weakly multiplicative. \label{Weakly Multiplicative}
		\item If $X_{i,j} \in \Domain$ for $1 \leq i \leq \ell$ and $1 \leq j \leq k_i$ and 
		\[
			\sum_{i = 1}^\ell \prod_{j = 1}^{k_i} X_{i,j} = 0,
		\]
		then
		\[
			\sum_{i = 1}^\ell \prod_{j = 1}^{k_i} \Sum \parent{X_{i,j}} = 0.
		\] \label{Weakly Multiplicative Equivalent Condition}
	\end{enumerate} 
\end{prop}

\begin{proof}
	$\ref{Weakly Multiplicative} \implies \ref{Weakly Multiplicative Equivalent Condition}$ Suppose that $(\Domain, \Sum)$ is weakly multiplicative, and let $(\Domain\prm, \Sum\prm)$ be any multiplicative extension of $(\Domain, \Sum)$. Then if $X_{i,j} \in \Domain \subseteq \Domain\prm$ for $1 \leq i \leq \ell$ and $1 \leq j \leq k_i$, we see $X = \sum_{i = 1}^\ell \prod_{j = 1}^{k_i} X_{i,j} \in \Domain\prm$, since $\Domain\prm$ is a ring. Then we compute 
	\[
		0 = \Sum\prm(0) = \Sum\prm(X) = \sum_{i = 1}^\ell \prod_{j = 1}^{k_i} \Sum\prm \parent{X_{i,j}} = \sum_{i = 1}^\ell \prod_{j = 1}^{k_i} \Sum \parent{X_{i,j}},
	\]
	which proves one direction of the result.
	
	$\ref{Weakly Multiplicative Equivalent Condition} \implies \ref{Weakly Multiplicative}$ Suppose \ref{Weakly Multiplicative Equivalent Condition} holds for $(\Domain, \Sum)$, and define $\parent{\M\Sum, M\Domain}$ in accordance with Definition \ref{Definition: Multiplicative Extension} above. I claim that $\M\Sum$ is well-defined. Suppose that 
	\[
		X = \sum_{i = 1}^\ell \prod_{j = 1}^{k_i} X_{i,j} = \sum_{i = 1}^{\ell\prm} \prod_{j = 1}^{k\prm_i} X\prm_{i,j}
	\] 
	for some $X_{i,j}, X\prm_{i,j} \in \Domain$. Taking differences, we may assume without loss of generality that 
	\[
		\sum_{i = 1}^\ell \prod_{j = 1}^{k_i} X_{i,j} = 0.
	\] 
	Then condition \ref{Weakly Multiplicative Equivalent Condition} tells us that $\M\Sum \parent X$ is well-defined; this being the case, $\M\Sum$ is clearly multiplicative.
\end{proof}

For any weakly multiplicative summation $(\Domain, \Sum)$ and any natural number $k \geq 1$, we define 
\[
	\M_k \Domain \coloneqq \Set{X = \sum_{i = 1}^\ell \prod_{j = 1}^{k} X_{i,j} \in \R[[\sigma]] \ : \ \ell \in \bbN \ \text{and} \ X_{i,j} \in \Domain},
\] 
and define $\M_k \Sum$ to be the restriction of $\M \Sum$ to $\M_k \Domain$. This gives us a grading 
\begin{equation*}
\Sum = \M_1\Sum \subseteq \M_2\Sum \subseteq \M_3\Sum \subseteq \cdots \subseteq \M\Sum.
\end{equation*}
Naturally, this grading is trivial if $\Sum$ is multiplicative. The familiar weakly multiplicative summation $\Sum_c$ exhibits this grading nicely. 

\begin{proposition}\label{MConvOrder} If a series $X = \sum_{n=0}^\infty x_n\sigma^n$ is in $\M_{k+1} \Domain_c$, then $x_n = O(n^k)$.
  If $X$ is in $\M\Sum_c$, then $x_n = O(n^k)$ for some $k$.
\end{proposition}

\begin{proof} 
	If $X$ converges, $x_n = O(1)$. As the $n$th term of an $(k+1)$-fold product is the sum of $\binom{k+n}{k}$ products of individual terms, the coefficients of the product of $k+1$ convergent series must be $O(n^k)$. The second claim follows immediately.
\end{proof} 

Mertens' Theorem \cite[page 321]{Knopp} (see also \cite[page 228]{Hardy}) says that the product of an absolutely convergent series and a convergent series is always convergent, and a result of Schur \cite{Schur} says that this property characterizes convergent series. Thus it makes sense to take $\M_0 \Sum_c \coloneqq \Sum_a$. For a weakly multiplicative summation $(\Domain, \Sum)$, we may define 
\[
\M_0 \Domain \coloneqq \Set{A \in \Domain \ : \ \text{if } X \in \Domain \ \text{then} \ A X \in \Domain},
\]
and $\M_0 \Sum \coloneqq \Sum |_{\M_0 \Domain}$. The set $\M_0 \Domain$ is easily seen to be shift-invariant and an algebra, so $\M_0 \Sum$ is a multiplicative summation.  

There is an intriguing similarity between this and the definition of finite-order Ces\`{a}ro means (see \cite[pages 96-98]{Hardy}). The product of an order-$k$ Ces\`{a}ro-H\"{o}lder summable series and an order-$\ell$ Ces\`{a}ro-H\"{o}lder summable series is Ces\`{a}ro-H\"{o}lder summable of order $k+\ell+1$, so the (domains of the) Ces\`{a}ro-H\"{o}lder means may be considered as an algebra graded by order plus 1.  The 0th order Ces\`{a}ro-H\"{o}lder mean is traditionally defined to be $\Sum_c$ itself.  Moreover, Hardy defines the Ces\`{a}ro mean of order -1 to be the restriction of $\Sum_c$ to series with terms $o(n^{-1})$, making the $k=-1$ case Mertens-like. It follows that every product of $k$ convergent series is Ces\`{a}ro-H\"{o}lder summable of order $k-1$.

\begin{question}
Is every Ces\`{a}ro-H\"{o}lder summable series a product of convergent series?
\end{question}

This fits naturally into the classical theory of divergent series, so we would have expected to find an answer in the literature, but we failed to find it. Again, the following result should be known.

\begin{proposition}\label{MConvGrade}
	The grading of $\M\Sum_c$ does not stabilize: for arbitrarily large $k$, there are series which are products of $k$ convergent series but no fewer.
\end{proposition}

\begin{proof} 
	As shown above, the terms of any series in $\M_{k+1}\Sum_c$ are $O(n^k)$. But for any $k$, the summation $\M\Sum_c$ contains series that are not $O(n^k)$. 
	
	Set
	\begin{equation*}
		X \coloneqq \frac{1}{\sqrt{1 + \sigma}} = \sum_{n=0}^{\infty} \parent{\frac{-1}{4}}^n \binom{2n}{n}\sigma^n =  1-\frac{1}{2}+\frac{3}{8}-\frac{5}{16}+ \frac{35}{128}- \frac{63}{256}+ \frac{231}{1024}-\cdots
	\end{equation*}
	It is known \cite[Problem 9.60]{Graham-Oren} that 
	\begin{equation*}
	4^{-n}\binom{2n}{n} = \frac{1}{\sqrt{\pi n}} \left(1-\frac{c_n}{n}\right), \mbox{ for } \frac{1}{9} < c_n <\frac{1}{8};
	\end{equation*} 
	thus the $n$th coefficient of $X$ is $O(n^{-1/2})$ and $X$ is conditionally convergent. By Abel's theorem, it converges to $(1+1)^{-1/2} = \frac{1}{\sqrt{2}}$. 

But $X^2$ is the Taylor series of $(1+x)^{-1}$ evaluated at $x=1$, that is
\begin{equation*}
X^2 = \sum_{n=0}^\infty (-1)^n \sigma^n = 1-1+1-1+\cdots
\end{equation*}
which is not convergent. Thus $X^2\in \M_2 \Domain_c$ but $X^2 \not\in \Domain_c$, and $\M\Sum_c(X^2) = 1/2$, the same value given by $\Sum_A$ and $\T\Add$.  Moreover, the $n$th term of $(X^2)^{k+1}$ is $(-1)^n$ times the number of ordered $(k+1)$-tuples $(n_0,n_1,\dots,n_k)$ such that $n_j \geq 0$ for each $j$, and $\sum_{j=0}^k n_j = n$. By a well-known counting trick, or invoking the Binomial Theorem, we get 
$$
X^{2k+2} = \sum_{n=0}^\infty (-1)^n \binom{k+n}{k} \sigma^n
$$
which has terms asymptotically of order $n^{k-1}$, so $X^{2k+2}$ cannot be a product of fewer than $k$ convergent series. Finally, if we had $\M_k \Sum = \M_{k+1}\Sum$, it would follow that all higher grades were also equal; so we have
\begin{equation*}
	\Sum_c =\M_1\Sum_c \subset \M_2\Sum_c \subset\cdots \subset \M_k\Sum_c \subset\cdots \subset \M\Sum.
\end{equation*}
\end{proof}

However, we can construct a weakly multiplicative summation which stabilizes at any desired grade.

\begin{example}
	Let 
	\[
		V \coloneqq \sum_{n = 0}^\infty (-1)^n {{\frac{1}{k+1}} \choose n} \sigma^n \in\bbC[[\sigma]],
	\]
	so that $V^{k+1} = 1 - \sigma$. For $0 \leq j \leq k$ let
	\[
		\Domain_j \coloneqq \R[V,\sigma] = \Domain \oplus \Domain \cdot V \oplus \Domain \cdot V^2 \oplus \dots \oplus V^j,
	\]
	with 
	\[
	\Sum_j : F_0+F_1V + F_2V^2+ \cdots +F_jV^j \mapsto \Add(F_0).
	\]
	Then $\M_j \Domain_1 = \Domain_{j+1} \supset \Domain_j$ and $\M_j \Sum_1 = \Sum_{j+1} \supset \Sum_j$ for $j < k$, but 
\[
\M_{k-1}\Domain_1 = \M_{k}\Domain_1 = \Domain_{k} \ \text{and} \ \M_{k-1}\Sum_1 = \M_{k}\Sum_1 = \Sum_{k}.
\]
\end{example}

We now consider how $\M$ interacts with $\T$. While artificial examples could be constructed, the classical summation $\Sum_c$ provides all the examples we require.

\begin{example}\label{TelNotMult} 
	The series 
	\[
	G_2 \coloneqq \frac{1}{1 - 2 \sigma} = \sum_{n=0}^\infty 2^n \sigma^n = 1 + 2 + 4 + 8 + 16 + 32 + \dots
	\]
	is in $\T\Sum_c$ but (as its terms grow exponentially) not in $\M\Sum_c$. 
\end{example}

Although $\T \Convsum$ is not multiplicatively closed, the following proposition shows that $\T \Convsum$ is at least ``multiplicatively closed relative to $\Sum_a$''.

\begin{proposition} 
	Let $B,C \in \Convdomain$ with $C \in \T\Absdomain$. Then $BC \in \T\Convdomain$.
\end{proposition}

\begin{proof} 
	By hypothesis there exists $F$ with $CF=A$, $\Add(F)\neq 0$, and $A\in\Absdomain$. Then $BCF=BA$ and by Mertens' Theorem \cite[page 321]{Knopp} this is convergent.
\end{proof}

Unfortunately, not all convergent series are in the domain of $\T\Sum_a$.

\begin{example}\label{ConvNotTel} 
	The series $C \coloneqq \sum_{n=1}^\infty \frac{(-1)^n}{n} \sigma^{n^2}$ is convergent but not telescopable over $\Sum_a$. Indeed, let $F = \sum_{n=0}^N f_n \sigma^n$ be a polynomial; assume without loss of generality that $f_0$ is nonzero and $f_N$ is the highest-order nonzero coefficient. Then for $n>N/2$ the $n^2$th coefficient of $FC$ has absolute value $\frac{|f_0|}{n}$ and the sum $\sum_{n=1}^\infty \frac{\abs{f_0}}{n}$ diverges.
\end{example} 

The same idea, modified to give a tidier square, shows that $\M\Sum_c \not\subset \T\Sum_c$.

\begin{example}\label{MultNotTel} 
Let
\[
L \coloneqq \bigcup_{n=0}^\infty\left \{m \cdot 2^{2^n} \ : \ 0\leq m \leq 2^{2^n}) \right\} = \{0,2,4,8,12,16,32,48,\dots \},
\]
and let $\ell_n$ be the $n$th element of $L$. We define
\[
	X \coloneqq \sum_{n=1}^\infty \frac{(-1)^n}{\log(n+2)}\sigma^{\ell_n}.
\]
For legibility, let $g(n) \coloneqq 2^{2^n}$. The $g(n)$th coefficient of $X^2 = \sum_{n = 0}^\infty x_n \sigma^n$ is given by

\begin{equation*}
x_{g(n)} = \sum_{\ell_i+\ell_j=g(n)} \frac{1}{\log(i+2)\log(j+2)}.
\end{equation*}
The pairs of indices in the sum are 
\[
\ell_i = kg(n-1), \ \ell_j = g(n)-k g(n-1)) \ \text{for} \ 0\leq k \leq g(n-1);
\]
and there are $g(n-1)$ of these terms. If $\ell_i \leq g(n)$, an easy induction shows that 
\[
g(n-1)< i <2g(n-1)
\]
and so
\[
	\log(i+2) < \log(2^{2^{n-1}+1}+2) < 2^{n-1}.
\]

Thus,
\[
x_{g(n)} > \frac{g(n-1)}{2^{2n-2}} = 2^{2^{n-1}-2n+2} \geq 2 \mbox{ for all } n.
\]

Moreover, there are no pairs $(\ell_i, \ell_j) \in L \times L$ with $g(n)-g(n-2) < \ell_i + \ell_j < g(n)$. Suppose by way of contradiction that $X^2$ telescopes over $\Sum_c$, say with $F X^2 \in \Convdomain$. Without loss of generality, we may take $F \in \bbC[\sigma]$ with $F(0) \neq 0$: now if we let $n$ be large enough that $g(n-2) > \deg(F)$,  then the $g(n)$th term of $F X^2$ is $F(0) x_{g(n)}$ and $F X^2$ does not converge, a contradiction.
\end{example}

It follows that neither of $\T\Sum_c$ and $\M\Sum_c$ contains the other. However, we have the following general result.

\begin{theorem}\label{Theorem: Telescopic-multiplicative extension}
	The functor of summations
	\[
		\T\M : \wMSumCat \to \MSumCat
	\]
	is an idempotent extension map which subsumes $\M$, $\T$, and $\M \T$.
\end{theorem}

\begin{proof}
	We first prove that $\T \M \T \Domain = \T \M \Domain$ for every weakly multiplicative summation $(\Domain, \Sum)$. As $\T$ is an extension map, it is clear that $\T \M \T \Domain \supseteq \T \M \Domain$. Suppose now that $X \in \T \M \T \Domain$. By Definition \ref{Definition: Telescopic Extension}, we may choose
	\[
	F \in \R[\sigma], \ A \in \M \T \Domain, \ f \in \Reg \Codomain, \text{and} \ x \in \Codomain
	\]
	so that 
	\[
	F X = A, \ \Sum(F) = f, \text{and} \ \M\T\Sum(A) = f x.
	\]
	But if $A \in \M \T \Domain$, then by Definition \ref{Definition: Multiplicative Extension}, we may write $A = \sum_{i=1}^\ell \prod_{j=1}^{k_i} X_{i, j} \in \M \T\Domain$ with each $X_{i, j} \in \T\Domain$. Again by Definition \ref{Definition: Telescopic Extension}, for each $i$ and $j$ we may choose 
	\[
	F_{i, j} \in \R[\sigma], \ A_{i, j} \in \Domain, \ f_{i, j} \in \Reg \Codomain, \ \text{and} \ x_{i, j} \in \Codomain
	\]
	so that 
	\[
	F_{i, j} X_{i, j} = A_{i, j}, \ \Sum(F_{i, j}) = f_{i, j}, \ \text{and} \ \Sum(A_{i, j}) = f_{i, j} x_{i, j}. 
	\]
	Now let 
	\begin{align*}
	F\prm &\coloneqq F \cdot \prod_{i=1}^\ell \prod_{j=1}^{k_i} F_{i, j} \ \text{and} \\
	A\prm &\coloneqq \sum_{i=1}^\ell \prod_{j=1}^{k_i} A_{i, j} \prod_{(i\prm, j\prm) \neq (i, j)} F_{i\prm, j\prm}.
	\end{align*}
	As a product of polynomials, $F\prm$ is a polynomial, and as $\Sum(F)$ and each $\Sum(F_{i, j})$ is regular, $\Sum(F\prm)$ is regular. By construction, $A\prm \in \M\Domain$, $F\prm X = A\prm$, and $\M\Sum(A\prm) = f\prm x,$ where $f\prm \coloneqq \Sum(F)$ and $x = \sum_{i=1}^\ell \prod_{j=1}^{k_i} x_{i, j}$ is as above. Thus $\T \M \T \Domain = \T \M \Domain$ as desired.
	
	We have shown $\T \M \T = \T \M$. On the other hand, $\T \T \M = \T \M$ because $\T$ is idempotent, so $\T \M$ subsumes $\T$. We now prove $\T \M$ subsumes $\M$. If $\Sum$ is weakly multiplicative then by Proposition \ref{Proposition: Telescopic extensions preserve multiplication}, $\T \M \Sum$ is multiplicative, and so $\M \T \M \Sum = \T \M \Sum$, since by Theorem \ref{Theorem: Multiplicative Extension}, $\M$ leaves multiplicative summations unchanged. On the other hand, $\T \M \M = \T \M$ because $\M$ is idempotent, so $\T \M$ subsumes $\M$. But as $\T \M$ subsumes $\M$ and $\T$, it is immediate that $\T \M$ subsumes $\M \T$ and itself, and our claim follows.
\end{proof}

Theorem \ref{Theorem: Telescopic-multiplicative extension} implies that $\M \T \Sum \subseteq \T \M \Sum$. If the codomain of $\Sum$ is a field, the reverse inclusion also holds.

\begin{prop}\label{Proposition: MTS = TMS if E is a field}
	Let $(\Domain, \Sum) \in \WSums \R \Codomain$ be a weakly multiplicative summation. If $\Codomain$ is a field them $\M \T \Sum = \T \M \Sum$.
\end{prop}

\begin{proof}
	Let $(\Domain, \Sum) \in \WSums \R \Codomain$ be a weakly multiplicative summation with $\Codomain$ a field. Theorem \ref{Theorem: Telescopic-multiplicative extension} tells us that $\M \T \Domain \subseteq \T \M \Domain$, so it suffices to prove the reverse inclusion.
	
	Let $X \in \T \M \Domain$. By Definition \ref{Definition: Telescopic Extension}, we may choose
	\[
	F \in \R[\sigma], \ A \in \M \Domain, \ f \in \Reg \Codomain, \text{and} \ x \in \Codomain
	\]
	so that 
	\[
	F X = A, \ \Sum(F) = f, \text{and} \ \M\Sum(A) = f x.
	\]
	But if $A \in \M \Domain$, then by Definition \ref{Definition: Multiplicative Extension}, we may write $A = \sum_{i=1}^\ell \prod_{j=1}^{k_i} A_{i, j} \in \M \Domain$ with each $A_{i, j} \in \Domain$. Now write $F = \sigma^m F\prm$ with $F\prm(0) \neq 0$, so that $F\prm(\sigma)$ is a unit in $\Codomain[[\sigma]]$. Then
	\[
	\sigma^m X = \sum_{i=1}^\ell \frac{A_{i, 1}}{F} \prod_{j=2}^{k_i} A_{i, j} \in \M\T\Domain.
	\]
	But as $\sigma^m X \in \M\T\Domain$ and $\M\T\Sum$ is a summation, we see $X \in \M \T\Domain$ as desired.
\end{proof}

Proposition \ref{Proposition: MTS = TMS if E is a field} notwithstanding, we pose the following question.

\begin{quest}
	Does $\M \T = \T \M$?
\end{quest}

We conjecture that the answer is negative.

\begin{example}\label{Example: MT != TM}
	Let $\R \coloneqq \bbC[z]$, let $\Codomain \coloneqq \bbC[x, y]$, and make $\Codomain$ an $\R$-algebra via $z \mapsto xy$. Let $X, Y \in \R[[\sigma]]$ be algebraically independent over $\R[\sigma]$. Let $(\Domain, \Sum)$ be the minimal extension of $\Add$ such that $\Sum : X \to x$ and $\Sum : Y \to y$. Thus
	\[
	\Domain = \R[\sigma] \oplus \R[\sigma] \cdot X \oplus \R[\sigma] \cdot Y.
	\]
	Now let $Z \coloneqq \frac{XY}{1 - \sigma + z \sigma^2}$. We readily verify that $Z \in \T \M \Domain$, and compute $\T \M \Sum (Z) = 1$. However, we suspect that $Z \not\in \M\T\Domain$, and thus $\M \T \neq \T \M$. If so, this also implies that $\M \T$ is not idempotent and does not subsume $\M$ or $\T$.

\end{example}

\section{Rational Extensions}\label{Section: Rational Extensions}

Many summations (most notably $\Convsum$) can be considered as the composition $\Lambda\Sigma$ of two linear functions (see \cite{Boos}), where $\Sigma$ takes a series to its sequence of partial sums (in $\Codomain$), and $\Lambda : \Codomain[[\sigma]] \to \Codomain$ is some limit-like operator assigning values in $\Codomain$ to sequences. Of course, the domain of $\Lambda$ is never all of $\Codomain [[\sigma]]$. For $\Lambda \Sigma$ to restrict to addition on finitely-supported series, it is necessary and sufficient that $\Lambda$ should take all eventually constant sequences to their constant values.

If $\Codomain$ is a field, it is often possible to extend a $\Lambda \Sigma$ summation via a \defi{N{\o}rlund mean}. Fix $P \in \Domain$ such that $\Sum P \neq 0$ and $(\Sigma P)_n \neq 0$ in $\Codomain$ for all $n$. For any sequence $S \in \R[[\sigma]]$, let $\Nor_P(S) \in \Codomain[[\sigma]]$ be the sequence whose terms are $(P S)_n / (\Sigma P)_n \in \Codomain$; if $\Lambda\Nor_P\Sigma(X)$ exists and equals $x$, we define $\Sum_P (X) = x$. This method of summation is closely related to telescoping. Indeed, we may define $\T_P(X)$ to be the series whose terms are $(P X)_n/\Sum (P)$ (in $\Codomain$); thus $\T_P (X) = PX  \Sum (P) \in \Codomain[[\sigma]]$, and $\Sum \T_P$ telescopes $X$ against the series $P$, and against no other series. The following proposition is adapted from Proposition 1.2 of \cite{Dawson}.

\begin{proposition}\label{Proposition: Lambda Sigma TelP = Lambda NorP Sigma}
	Suppose $\Codomain$ is a field. Let $(\Domain, \Sum) \in \Sums \R \Codomain$ be a summation with $\Sum = \Lambda \Sigma$ as above. Suppose that $P \in \Domain$ is as above, and either $\Lambda$ preserves (termwise) products, or $P$ is finitely supported and $\Lambda$ takes every eventually constant sequence to its final value. Then 
	\begin{equation*}
		\Lambda\Nor_P\Sigma(X) = \Lambda\Sigma\T_P(X)
	\end{equation*}
	whenever the left-hand side or the right-hand side is defined.
\end{proposition}

If $P$ is finitely supported and $\Sum(PX)$ is defined, then $X \in \T\Domain$ and $\Lambda\Sigma\T_P \parent X = \T\Sum \parent X$ \cite{Dawson}. If $\Lambda$ preserves products, but $P$ is not finitely supported, we have a generalized form of telescoping. These remarks motivate the following definition.

\begin{defn}\label{Definition: Rational Extension}
	For a multiplicative summation $(\Domain, \Sum) \in \MSums \R \Codomain$, the \defi{rational extension} $(\Q\Domain, \Q\Sum)$ of $(\Domain, \Sum)$ is defined as follows. For $X \in \R[[\sigma]]$, we say $X \in \Q\Domain$ if there exists $A, B \in \Domain,$ $b \in \Reg \Codomain$, and $x \in \Codomain$ such that $A = BX$, $\Sum(B) = b$, and $\Sum(A) = b x$. We define
	\begin{align*}
		\Q\Sum &: \Q\Domain \to \Codomain, \\
		\Q\Sum &: X \mapsto x \ \text{if} \ X \ \text{is as above.}
	\end{align*}
	We extend this definition to weakly multiplicative summations $(\Domain, \Sum) \in \WSums \R \Codomain$ by setting $(\Q\Domain, \Q\Sum) \coloneqq (\Q\M\Domain, \Q\M\Sum)$.
\end{defn}

Repeating (essentially) the proof of Theorem \ref{Theorem: Telescopic Extension}, we obtain:

\begin{theorem}\label{Theorem: Rational Extension}
	The functor of summations
	\[
		\Q : \wMSumCat \to \MSumCat
	\]
	is an idempotent extension map which subsumes $\M$, $\T$, $\M \T$, and $\T \M$.
\end{theorem}

If $\Sum = \Q\Sum$, so that $\Sum$ is $\Q$-closed, we say $\Sum$ is \defi{rationally closed}. 

\begin{obsrem}
	The collection $\QSumCat$ of rationally closed multiplicative summations forms a full subcategory of $\wMSumCat$. Observe $\Q$ is left adjoint to the forgetful functor from $\QSumCat$ to $\wMSumCat$.  We could define the \defi{rational fulfillment} of a weakly multiplicative summation $\Sum$ to be the largest extension of $\Sum$ which is compatible with every extension of $\Sum$ in $\QSumCat$, but the rational fulfillment of a summation is the same as its multiplicative fulfillment, so such terminology would be superfluous.
\end{obsrem}

The rational extension, then, is defined on formal quotients $A/B$ with coefficients in $\R$, where $A$ is summable and $B$ has nonzero sum. For comparison, the telescopic extension is defined on all formal quotients $A/F$ with coefficients in $\R$, where $A$ is summable and $F$ is \emph{finitely supported} with nonzero sum. How meaningful is this difference? We readily observe that $\T\Add = \Q\Add$, where $\Add$ is the summation with domain the finitely-supported series. Likewise, as $\Q$ is idempotent and subsumes $\T$, we see that $\T\Sum = \Q\Sum$ for every rationally closed summation $\Sum$. We now give a less trivial example of a summation for which the telescopic and rational extensions are equal.

\begin{definition}\label{GeomConv}
A series $S$ with coefficients $s_n \in \bbC$ will be called \defi{geometrically convergent} if there exists $r \in (0,1)$ such that $s_n = O(r^n)$ as $n \to \infty$. We write $(\Domain_g, \Sum_g)$ for the restriction of $(\Domain_c, \Convsum)$ to geometrically convergent series.
\end{definition}

If $S$ is geometrically convergent, the radius of convergence of $S(z)$ is at least $1/r$, which is strictly greater than 1. Note that every geometrically convergent series is absolutely convergent, that $\Sum_g$ is multiplicative, and that $\Domain_g$ is a subalgebra of $\Domain_a$.

\begin{proposition}
	If $S = S(\sigma) \in \bbC[[\sigma]]$ is geometrically convergent, then for some $\epsilon > 0$, the function $S(z)$ has only finitely many zeros on some $(1+\epsilon)$-disc. There is a polynomial $P_S(z)$ that has the same zeros as $S$ (including multiplicity) on that disc; in particular, if $\Sum_g(S) = S(1) \neq 0$, then $\Add(P_S) = P_S(1) \neq 0$. Moreover, $P_S(z)/S(z)$ has only removable singularities in a $(1+\epsilon)$-disc; and the series $T$ of Taylor coefficients of $P_S(z)/S(z)$ is geometrically convergent.
\end{proposition}

\begin{proof} 
	If $S$ is geometrically convergent, $S(z)$ has only finitely many zeros on $\{z \ : \ \abs{z} \leq \frac{1+r}{2r}\}$; otherwise the set of zeros would have a cluster point $z_0$ with $\abs{z_0} < \frac{1}{r}$. Such a cluster point would be an essential singularity of $S(z)$ that lies within the disc of convergence, which is impossible. The other claims follow immediately.
\end{proof}

\begin{corollary}
$\T\Sum_g=\Q\Sum_g$.
\end{corollary}
\begin{proof}
	If $S',S \in \Sum_g$, then $P_S\cdot(S'/S) = TS' \in \Sum_g$. Thus, $S'/S$ is in fact telescopable over $\Sum_g$.
\end{proof}

However, there are many examples of multiplicative or weakly multiplicative  summations for which $\T\Sum \neq \Q\Sum$. 

\begin{example}
	If $X$ is transcendental over $R[\sigma]$, we extend the domain of $\Add$ by $X$, mapped to any unit $\xi$. The resulting multiplicative summation $\Add[X]$ sums precisely the series of the form $\sum_{j=0}^\ell F_j X^j$, where $F_j \in \R[\sigma]$. Then $\T \R[\sigma, X]$ does not assign a value to $X^{-1}$, but $\Q \R[\sigma, X]$ does.
\end{example}

We will now construct a more natural example, a series in $\Q\Domain_a$ which cannot be telescoped even over $\Sum_c$. Consider the (complex-valued) function 
\begin{align*}
S &: \bbC \setminus \set{-1} \to \bbC, \\
S &: z \mapsto (1 + z)(1 + \exp(z/(1+z)));
\end{align*}
this function the value $0$  at the points of the set $\{(2k+1)\pi i/(1-(2k+1)\pi i) \ : \ k \in \bbZ\}$, which lies on the circle $\{z \ : \ \abs{z-\frac{1}{2}} =\frac{1}{2}\}$ and has a cluster point at $-1$ (Figure \ref{zeros}). This cluster point is necessarily an essential singularity.  The key to our construction is that, despite this, the Taylor series for the function $S(z)$ is absolutely convergent at 1. 

\begin{figure}[h]
	\begin{center}
	\includegraphics[width=4cm]{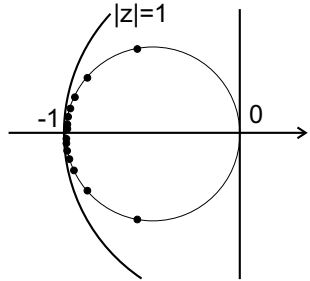}
	\caption{The set $\{(2k+1)\pi i/(1-(2k+1)\pi i) \ : \ k \in \bbZ\}$}
	\label{zeros}
\end{center}
\end{figure}

Define the auxiliary function 
\begin{align*}
T &: \bbC \setminus \set{1} \to \bbC, \\
T &: z \mapsto (1 - z)\exp(-z/(1-z)),
\end{align*}
so that $S(z) = 1 + z + T(-z)$.

Thus 
\[
T(z) = 1 - 2z + \frac{z^2}{2} + \frac{z^3}{3} + \frac{5z^4}{24} + \frac{7z^5}{60} + \frac{37z^6}{720} + \frac{17z^7}{2520} - \frac{887z^8}{40320} + \cdots ,
\]
and
\[
S(z) = 2 + 3z + \frac{z^2}{2}- \frac{z^3}{3} + \frac{5z^4}{24} - \frac{7z^5}{60} + \frac{37z^6}{720} - \frac{17z^7}{2520} - \frac{887z^8}{40320} \cdots.
\]

The signs of the coefficients $t_n$ (and of $s_n$) are not strictly alternating. Nor are the absolute values monotone decreasing: rather, they exhibit an oscillation that both lengthens and decays (see Figure \ref{plot}). The curve shown is $y=0.9293 n^{-5/4}$ (the constant was found experimentally.) 

\begin{figure}[h]
     \begin{center}
	\includegraphics[width=14cm]{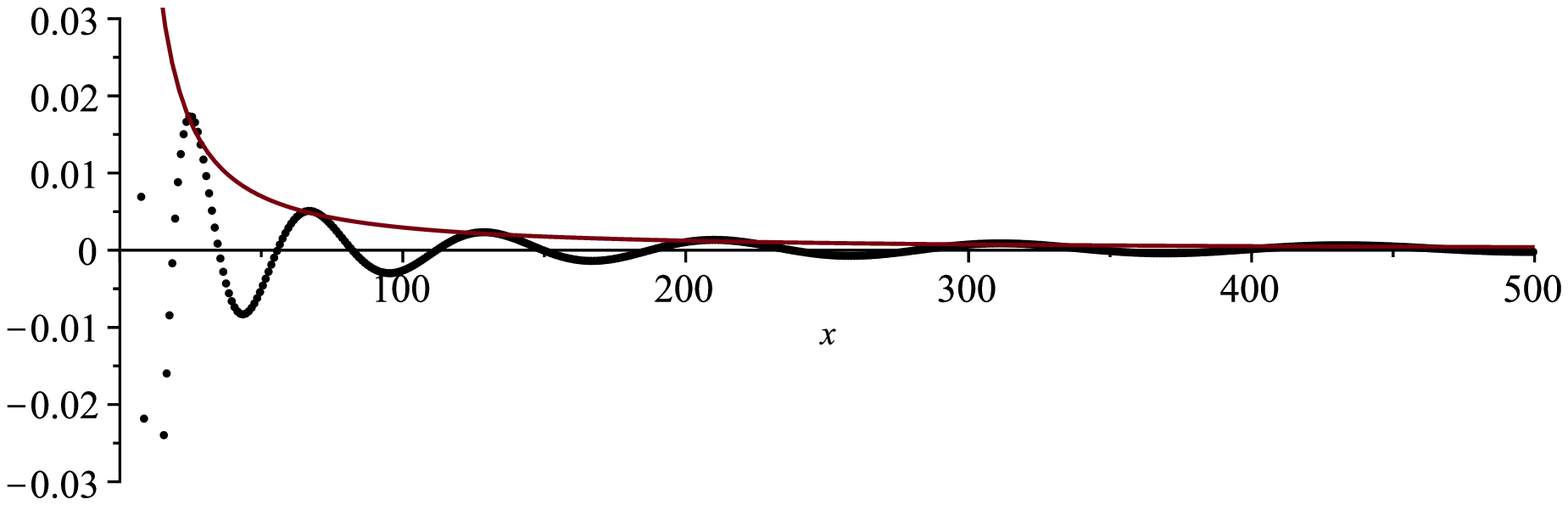}
	\caption{Taylor coefficients of $T(x) = (1-x)\exp\parent{\frac{-x}{1-x}}$}
	\label{plot}
   \end{center}
\end{figure}

\begin{prop}\label{Proposition: T(z) is absolutely convergent}
	The power series expansions for $T(z)$ and $S(z)$ converge for $\abs{z} \leq 1$.
\end{prop}

\begin{proof}
	As $S(z) = 1 + z + T(-z)$, it suffices to prove the claim for $T(z) = \sum_{n = 0}^\infty t_n z^n$. Note that $T(z) = (1 - z) \cdot e \cdot \exp(-1/(1-z))$, so letting $r = 0$ in the asymptotic given by \cite[Theorem 2]{Wright}, we see
	\[
	t_n = O(n^{-5/4}).
	\]
	The series $\sum_{n = 1}^\infty n^{-5/4}$ is absolutely convergent, so $T(z)$ is absolutely convergent on and within the unit circle, as desired.
\end{proof}

\begin{remark}
	Wright's asymptotic for the coefficients of $T(x)$ applies to a much more general class of functions, and is much more refined than the statement $t_n = O(n^{-5/4})$ we used in the proof of Proposition \ref{Proposition: T(z) is absolutely convergent}. Indeed, it suffices to explain the oscillating behavior shown in Figure \ref{plot}. However, this coarse statement suffices for our purposes.
\end{remark}

\begin{example}\label{RatNotTel} Let $X(\sigma) \coloneqq 1/S(\sigma)$, so that
\[
X(\sigma) = \frac{1}{2} - \frac{3}{4} + 1 - \frac{59}{48}  + \frac{17}{12}  - \frac{247}{160}  + \frac{383}{240}  
- \frac{126151}{80640}  + \frac{14599}{10080} +\cdots.
\]
is rational over $\Sum_a$. The series $X$ it is not telescopable; its generating function $X(z)$ has infinitely many poles on the open unit disc, and multiplying by a nonzero polynomial $P(z)$ can only remove finitely many of these poles, so $PX$ is not even convergent. Thus $\T\Sum_a \neq \Q\Sum_a$, and $\T\Sum_c \neq \Q\Sum_c$.
\end{example}

We close with an observation about the image of $\Q\Sum$. This perspective will be revisited in \cite{Dawson-Molnar-3}.

\begin{prop}\label{Proposition: The image of Q(Sum) is the field of fractions of Sum(Domain)}
	Let $\Codomain$ be a field. For any multiplicative summation $(\Domain, \Sum) \in \MSums \R \Codomain$, the image of $\Q \Domain$ under $\Q\Sum$ is the field of fractions of $\Sum(\Domain)$.
\end{prop}

\begin{proof}
	By Theorem \ref{Theorem: Rational Extension}, $\Q\Sum$ is an $\R$-algebra homomorphism, and so its image is certainly a ring. Now let $x \neq 0$ be in the image of $\Q\Sum$, and write $x = \Q\Sum \parent X$ with $X \in \Q\Domain$. As $X \in \Q\Domain$, there exist series $A, B \in \Domain$ with $A = B X$ and $\Sum \parent A / \Sum \parent B = x$: this shows that the image of $\Q\Sum$ is contained in the field of fractions of $\Sum(\Domain)$.  Replacing $X$ with $1 - \sigma + \sigma^2 X \in \Domain$ if necessary, we may assume $X$ is a unit. We see $B = A X\inv$ and $\Sum \parent B / \Sum \parent A = x\inv$, and so $X\inv \in \Q \Domain$ with $\Q \Sum (X\inv)  = x\inv$. The claim follows. 
\end{proof}

\begin{corollary}
	Let $\Codomain$ be a field. For any multiplicatively fulfilled multiplicative summation $(\Domain, \Sum) \in \MSums \R \Codomain$, the image of $\Domain$ under $\Sum$ is a subfield of $\Codomain$.
\end{corollary}

\bibliographystyle{amsplain}

\end{document}